\documentclass[10pt]{article}

\usepackage{graphicx}
\usepackage{subfig}
\usepackage{caption}
\usepackage{booktabs,multirow}
\usepackage{amssymb,amsmath,amsthm}
\usepackage[numbers]{natbib}
\usepackage{tikz}
\usepackage{subfig}
\usepackage{mathrsfs}

\usepackage[all]{xy}
\usepackage{float}
\usepackage{subfig}
\usepackage{graphicx}
\usepackage[overload]{empheq}
\usepackage{algorithmic}

\usepackage{framed} 
\usepackage{lastpage}
\usepackage[algoruled]{algorithm2e}
\usepackage{color,soul}
\usepackage{authblk}

\usepackage[pagebackref=false,colorlinks,linkcolor=blue,citecolor=magenta]{hyperref}
\usepackage[nottoc]{tocbibind}
\newcommand{\bm}{\mathbf}
\let\oldnl\nl
\newcommand{\nonl}{\renewcommand{\nl}{\let\nl\oldnl}}

\newtheorem{theorem}{Theorem}[section]
\newtheorem{definition}{Definition}[section]

\newtheorem{lemma}[theorem]{Lemma}
\newtheorem{assumption}{Assumption}[section]
\newtheorem{remark}{Remark}[section]

\title{An effective subgradient algorithm via Mifflin's line search for nonsmooth nonconvex multiobjective optimization}

\author[1]{Morteza Maleknia}
\author[2]{Majid Soleimani-damaneh}

\affil[1]{School of Mathematics, Statistics and Computer Science, College of Science, University of
	Tehran, Tehran, Iran}
\affil[2]{Department of
	Mathematical Sciences, Isfahan University of Technology, Isfahan, Iran}

\affil[ ]{\textit {\{m.maleknia@iut.ac.ir,soleimani.d@ut.ac.ir\}}}

\begin{document}

\maketitle

	\section*{Abstract}
	We propose a descent subgradient algorithm for unconstrained nonsmooth nonconvex multiobjective optimization problems. To find a descent direction, we present an iterative process that efficiently approximates the (Goldstein) $\varepsilon$-subdifferential of each objective function. To this end, we develop a new variant of Mifflin's line search in which the subgradients are arbitrary and its finite convergence is proved under a semismooth assumption. In order to reduce the number of subgradient evaluations, we employ a backtracking line search that identifies the objectives requiring an improvement in the current approximation of the $\varepsilon$-subdifferential. Meanwhile, for the remaining objectives, new subgradients are not computed. Unlike bundle-type methods, the proposed approach can handle nonconvexity without the need for algorithmic or parametric adjustments. Moreover, the quadratic subproblems have a simple structure, and hence the method is easy to implement. We analyze the global convergence of the proposed method and prove that any accumulation point of the generated sequence satisfies a necessary Pareto optimality condition. Furthermore, our convergence analysis addresses a theoretical challenge in a recently developed subgradient method. By conducting numerical experiments, we observe the practical capability of the proposed method and evaluate its efficiency when applied to a diverse range of nonsmooth test problems.

	\section{Introduction}
	In many practical situations, one needs to make a decision under multiple conflicting criteria. For instance, consider a refinery that intends to maximize petrol production while minimizing toxic emissions at the same time. This problem is a decision-making problem with two conflicting objectives. Such a problem is called a \emph{\textbf{M}ultiobjective \textbf{O}ptimization \textbf{P}roblem} (MOP). In contrast with single objective problems, due to the conflicts among objectives, there is no solution to an MOP optimizing all  objectives simultaneously. Therefore, the optimality notion is replaced with 
	\emph{Pareto optimality}, which represents an optimal compromise among all objectives. In practical terms, it is of crucial importance to approximate the set of all Pareto optimal points, namely the \emph{Pareto front}, which is a challenging issue. This problem becomes even more complicated when objectives are not smooth and convex. 
	
	In this paper, we consider the following unconstrained multiobjective minimization problem:
	\begin{equation}\label{Main-Problem}
	\min \,\, \left(f_1(\bm x), f_2(\bm x),\ldots,f_p(\bm x)\right) \quad \text{s.t.} \quad \bm x\in\mathbb{R}^n,
	\end{equation} 
	where the objective functions $f_i:\mathbb{R}^n\to\mathbb{R}$, $i=1,\ldots,p$, are locally Lipschitz, but not necessarily smooth or convex. One of the main approaches to deal with MOP \eqref{Main-Problem} is to employ the scalarization methods \cite{Meitinent,Ehrgott2005,MO_Book}. This approach turns an MOP to a nonsmooth single-objective optimization problem. Then, a global optimal solution of the resulted single-objective problem provides a (weak) Pareto optimal point for the MOP. Regarding a nonsmooth single-objective optimization problem, various types of methods are developing rapidly. For instance, subgradient methods \cite{bagirov2012,bagirov2010}, bundle methods \cite{kiwielbook,inexactBundle1,Maleknia-Coap,Maleknia-Oms,BT-method}, and gradient sampling methods \cite{Burke2005,Kiwiel2007,Curtis2012,Maleknia-jota} are among the most common numerical approaches to minimize a locally Lipschitz function.  The weighted sum method \cite{Geof,MO-job} is widely recognized as one of the most popular scalarization methods. By adjusting the weighting vectors, it is possible to obtain the entire Pareto front in the convex case, at least in theory. However, in situations where the image space is not $\mathbb{R}^p_{\geq}$-convex, which occurs in nonconvex cases, a significant portion of the Pareto front cannot be achieved using this method (for a bi-objective example, see Subsection \ref{Subsec}). In addition, for some MOPs, even in the convex case, a broad range of weighting vectors lead to an unbounded single-objective problem (for a bi-objective instance, see \cite{MO2}). Similar limitations can also be observed with other aggregation techniques.

	As another class of methods for solving MOP \eqref{Main-Problem}, one can point to the iterative descent methods. These approaches are extensions of the existing descent methods in single-objective optimization to the multiobjective case, which can be directly applied to the non-scalarized  MOP \eqref{Main-Problem}. These methods are descent in the sense that, during each iteration, a descent direction that simultaneously reduces the value of all the objectives is found and utilized to update the current point. Due to the theoretical and practical challenges in developing descent methods for MOPs, the literature along this line of research is scarce, especially for problems with nonsmooth and nonconvex data. Some extensions of the subgradient methods that can be applied to MOPs with convex objectives can be found in \cite{Multi-sub1,Multi-sub2,Multi-sub3}. As a more recent subgradient method, Gebken and Peitz \cite{Multi-sub4} developed a subgradient method that can efficiently solve nonsmooth nonconvex MOP \eqref{Main-Problem}. In  \cite{techreport-multi}, M\"{a}kel\"{a} et al. presented  an effective generalization of the proximal bundle methods that can deal with MOP \eqref{Main-Problem}. A quasi-Newton method can be found in \cite{QUASI-MO}. Furthermore, some variants of the proximal bundle methods for solving nonsmooth nonconvex MOPs which are able to handle inexact information have been proposed in \cite{Multi-proximal1,Multi-proximal2}.
	
	In this paper, we develop a descent implementable subgradient algorithm to find Pareto optimal points of MOP \eqref{Main-Problem}. The main idea is to develop an iterative approach to efficiently approximate the (Goldstein) $\varepsilon$-subdifferential \cite{Bagirov2014} of each objective function. The least-norm element in the convex hull of the approximated $\varepsilon$-subdifferentials provides a trial search direction. Such an element is computed by solving a convex quadratic subproblem. Next, a limited backtracking line search determines whether we can take a serious step or we need to take a null step to improve our approximations of $\varepsilon$-subdifferentials. In this way, for the given positive sequences $\varepsilon_{\nu}\downarrow0$ and $\delta_{\nu}\downarrow0$, we generate a sequence of $\big(\delta_{\nu}, \mathcal{S}_{\varepsilon_{\nu}}(\bm x_{\nu+1})\big)$-substationary points (see Definition \ref{Def1}), such that any accumulation point of this sequence is Clarke substationary for MOP \eqref{Main-Problem}.
	
	 Regarding the details of the presented algorithm, the following features are highlighted.  The \emph{improvement function} which is used in M\"{a}kel\"{a} et al's proximal bundle method \cite{techreport-multi,Meitinent} poses an underlying issue. Indeed, once a null step occurs, the improvement function is not able to identify those objectives whose bundles of subgradient have not been enriched adequately. Therefore, after a null step, it is inevitable to compute a new subgradient for all the objectives. The same happens in recently developed proximal bundle methods \cite{Multi-proximal1,Multi-proximal2}. To resolve this difficulty, we propose a simple backtracking line search that excludes the objectives for which, after a null step, we do not need to enrich the bundle of information. For the remaining objectives, we develop a new variant of Mifflin's line search \cite{kiwielbook} that finds a new effective subgradient to improve the bundle of information. In this manner, one can significantly reduce the number of subgradient evaluations. Moreover, in Gebken and Peitz \cite{Multi-sub4}, the proposed line search procedure to find effective subgradients lacks the finite convergence property. More precisely, they merely argue that the asymptotic behavior of the proposed line search provides a subdifferential set within which the desired subgradient element exists (see Lemma 3.4 in \cite{Multi-sub4}). In contrast, our proposed line search exactly finds the desired subgradient element in finite number of iterations.
	 
	
	We study the global convergence of the proposed method and show that any accumulation point of the generated sequence is Clarke substationary for MOP \eqref{Main-Problem}. Furthermore, in Gebken and Peitz \cite{Multi-sub4}, convergence of the method to a Clarke substationary point through a sequence of $\big(\delta, \mathcal{S}_\varepsilon(\bm x)\big)$-substationary points remains as a theoretical challenge (see Remark 4.1 in \cite{Multi-sub4}). This problem is comprehensively solved in our convergence analysis. 
	 
	 By means of numerical tests, we show the efficiency and ability of the proposed method in practice. For this purpose, we first analyze the typical behavior of the method on a bi-objective test problem. Next, by using the multi-start technique, we assess the efficiency of our method in approximating the Pareto front. In our third experiment, we seek a sparse solution to an underdetermined linear system by minimizing a nonsmooth bi-objective model. Eventually, a set of nonsmooth convex
	 and nonconvex test problems are considered to provide some comparative results with the multiobjective proximal bundle method proposed by  M\"{a}kel\"{a} et al. \cite{techreport}.
	 
	 The rest of the paper  is organized as follows. Some required basic concepts are provided in Section \ref{Sec2}. Our main strategy for finding a descent direction along with finite convergence property of the proposed line search is given in Section \ref{Sec3}. Global convergence is studied in Section \ref{Sec4}. Numerical results are reported in Section \ref{Sec6}, and Section \ref{Sec7} concludes the paper.

	\label{Sec2}
\section{Mathematical preliminaries}
Throughout this paper, we use the following notations. The usual inner product in the Euclidean space $\mathbb{R}^n$  is denoted by $\bm x^T\bm y$, which induces the Euclidean norm $\lVert\bm x\rVert=(\bm x^T\bm x)^{1/2}$. An open ball with center $\bm x\in \mathbb{R}^n$ and radius $\varepsilon>0$ is denoted by $\mathcal B(\bm x, \varepsilon)$, that is,
$$\mathcal B(\bm x, \varepsilon):=\{\bm y\in\mathbb{R}^n \,\,:\,\, \lVert \bm y-\bm x\rVert< \varepsilon\}. $$
Moreover, $\mathbb{N}$ is the set of natural numbers, $\mathbb{N}_0:=\mathbb{N}\cup\{0\}$, and $\mathbb{R}_+:=(0,\infty).$

Suppose $f:\mathbb{R}^n\to\mathbb{R}$ is a locally Lipschitz function. Then, by Rademacher's theorem \cite{Evans2015}, $f$ is differentiable almost everywhere on $\mathbb{R}^n$. Let
$$\Omega_f:= \{\bm x\in\mathbb{R}^n \,\, : \,\, f\,\, \text{is not differentiable at}\,\, \bm x  \}.$$
Then,  the Clarke subdifferential of $f$ at a point $\bm x\in\mathbb{R}^n$ is defined as \cite{Clarke1990}
\begin{equation*}
\partial f(\bm x):= \texttt{conv} \{\boldsymbol{\xi}\in\mathbb{R}^n \,\,:\,\, \exists \, \{\bm x_i \}\subset\mathbb{R}^n \setminus\Omega_f\,\,\,\, \text{s.t.} \,\,\,\, \bm x_i\to \bm x \,\, \text{and} \,\, \nabla f(\bm x_i)\to \boldsymbol{\xi} \},
\end{equation*}
where $\texttt{conv}$ denotes the convex hull of a set. Furthermore, for any $\varepsilon> 0,$ the (Goldstein) $\varepsilon$-subdifferential of $f$ at a point $\bm x\in\mathbb{R}^n$ is the set  \cite{Bagirov2014}
$$\partial_\varepsilon f(\bm x):=\texttt {cl\,conv}\{\partial f(\bm y) \,\, : \,\, \bm y\in \mathcal{B}(\bm x, \varepsilon) \},$$
in which $\texttt{cl\,conv}$ is the closure of the convex hull.  For any $\varepsilon> 0$ and $\bm x\in\mathbb{R}^n$, the set $\partial_\varepsilon f(\bm x)$ is a nonempty, convex and compact subset of $\mathbb{R}^n$. If $f$ is differentiable at $\bm x\in\mathbb{R}^n$, then $\nabla f(\bm x)\in\partial f(\bm x)$. Furthermore, if $f$ is smooth at $\bm x\in\mathbb{R}^n$, we have $\partial f(\bm x)=\{\nabla f(\bm x)\}$.  Also, the set-valued map $\partial_\varepsilon f:\mathbb{R}^n\rightrightarrows\mathbb{R}^n$ is locally bounded and upper semicontinuous \cite{Clarke1990}. The generalized directional derivative of function $f$ at point $\bm x\in\mathbb{R}^n$ and direction $\bm d\in\mathbb{R}^n$ is given by \cite{Clarke1990}
\begin{equation*}
f^\circ(\bm x; \bm d):=\limsup_{\substack{\bm y\to \bm x\\t\downarrow 0}} \frac{f(\bm y+t\bm d)-f(\bm y)}{t}.
\end{equation*}
For the class of locally Lipschitz functions, $f^\circ(\bm x; \bm d)$ always exists for any $\bm x, \bm d\in\mathbb{R}^n$, and it is shown that $f^\circ(\bm x; \bm d)=\max\{\boldsymbol{\xi}^T\bm d \,\,: \,\, \boldsymbol{\xi}\in\partial f(\bm x)  \}$ \cite{Clarke1990}. 

The notion of \emph{Pareto optimality} is of crucial importance in multiobjective optimization. A point $\bm x\in\mathbb{R}^n$ is called Pareto optimal for problem \eqref{Main-Problem} if there is not any $\bm y\in\mathbb{R}^n$ such that $f_i(\bm y)\leq f_i(\bm x)$, for all $i=1,\ldots,p$, and $f_j(\bm y)<f_j(\bm x)$, for some $j\in\{1,\ldots,p\}$ \cite{Ehrgott2005}. It is recalled that for a point $\bm x\in\mathbb{R}^n$ to be a Pareto optimal, it is necessary that \cite{Meitinent}
 $$\bm 0\in\texttt{conv}\left\{ \cup_{i=1}^p\partial f_i(\bm x)\right\}.$$
A point $\bm x$ which satisfies the above condition is called a \emph{Clarke substationary} point. 

In the current paper, the notion of weakly upper semismoothness is important in designing an effective line search procedure. 

\begin{definition}[\cite{Mifflin-semismooth}]
	The locally Lipschitz function $f:\mathbb{R}^n\to\mathbb{R}$ is called  weakly upper semismooth if  for any $\bm z, \bm d\in\mathbb{R}^n$ and sequences $\{\boldsymbol{\xi}_i\}_i\subset\mathbb{R}^n$ and $\{h_i\}_i\subset\mathbb{R}_+$ satisfying $h_i\downarrow 0$ as $i\to\infty$ and $\boldsymbol{\xi}_i\in\partial f(\bm z+h_i\bm d)$, one has
	$$\limsup_{i\to\infty} \boldsymbol{\xi}_i^T \bm d\geq \liminf_{i\to\infty} \frac{f(\bm z+ h_i\bm d)-f(\bm z)}{h_i}.$$	
\end{definition} 

The class of weakly upper semismooth functions encompasses a wide range of functions, including convex, concave, and functions of max- and min-type \cite{Mifflin-semismooth,Mifflin,bagirov2020}. In fact, semismoothness is a popular assumption which is widely used in the context of general nonsmooth nonconvex optimization \cite{kiwielbook,BT-method,Zowe-book}.
%

\section{Finding a descent direction}\label{Sec3}
In this section, we develop an iterative procedure to find an efficient descent direction for problem \eqref{Main-Problem} at a given point $\bm x\in\mathbb{R}^n$. A direction $\bm d\in\mathbb{R}^n$ is called a descent direction for problem \eqref{Main-Problem} at point $\bm x\in\mathbb{R}^n$ if there exists $z>0$ such that
\begin{equation}\label{dscent-direction}
f_i(\bm x+t \bm d)-f_i(\bm x)<0, \quad \text{for all} \,\, t\in(0,z), \,\, \text{and} \,\, i=1,\ldots,p.
\end{equation}
The following lemma, which is essentially well-known, presents a theoretical approach to find a descent direction. For the sake of completeness, we provide a proof. Before it, let us denote the normal cone \cite{murdokhovich-Book1} of the convex set $C\subset\mathbb{R}^n$ at a point $\bm x\in C$ by $\mathcal{N}(C, \bm x)$.

\begin{lemma}
Assume $\boldsymbol{\xi}^*\neq \bm 0$ solves the following minimization problem:
\begin{equation}\label{steepest-descent}
\min \big\{ \lVert \boldsymbol{\xi}\rVert \,\, : \,\,  \boldsymbol{\xi}\in \text{\rm \texttt{conv}}\{ \cup_{i=1}^p\partial f_i(\bm x)\} \big\}.
\end{equation}
Then the normalized direction 
\begin{equation}\label{d_bar}
{\bar{\bm{d}}}:=-\boldsymbol{\xi}^*/\lVert \boldsymbol{\xi}^*\rVert,
\end{equation}
is a descent direction for multiobjective problem \eqref{Main-Problem} at point $\bm x$.	
\end{lemma} 
\begin{proof}
	Compactness of $\texttt{conv}\{ \cup_{i=1}^p\partial f_i(\bm x)\}$ yields the existence of $\boldsymbol{\xi}^*$. Since $\boldsymbol{\xi}^*$ solves  problem \eqref{steepest-descent}, we deduce \cite{Rockafellar2004}
	$$-\boldsymbol{\xi}^*\in \mathcal{N}\left(\texttt{conv} \{\cup_{i=1}^{p}\partial f_i(\bm x) \}, \boldsymbol{\xi}^*\right), $$
	yielding 
	\begin{equation*}\label{L0-1}
	\boldsymbol{\xi}^T \boldsymbol{\xi}^*\geq \lVert \boldsymbol{\xi}^*\rVert^2, \quad \text{for all} \,\, \boldsymbol{\xi}\in \texttt{conv} \{\cup_{i=1}^{p}\partial f_i(\bm x) \},
	\end{equation*}
	which is equivalent to
	\begin{equation*}\label{L0-2}
	-\boldsymbol{\xi}^T \frac{\boldsymbol{\xi}^*}{\lVert \boldsymbol{\xi}^* \rVert  }\leq - \lVert \boldsymbol{\xi}^*\rVert<0, \quad \text{for all} \,\, \boldsymbol{\xi}\in \texttt{conv} \{\cup_{i=1}^{p}\partial f_i(\bm x) \}.
	\end{equation*}
	Consequently, $\boldsymbol{\xi}^T\bm d<0$ for all $\boldsymbol{\xi}\in \texttt{conv} \{\cup_{i=1}^{p}\partial f_i(\bm x) \}$, which means $f^\circ_i(\bm x; \bm d)<0$, for $i=1,\ldots,p$, and the assertion follows immediately from the definition of $f^\circ(\bm x; \bm d)$.
\end{proof}

For some $i\in\{1,\ldots,p\}$, assume $\bm x\in\mathbb{R}^n$ is a continuously differentiable point for the objective $f_i$, which is close  to a nonsmooth region. In this case, the subdifferential set  $\partial f_i(\bm x)$ coincides with the singleton $\{\nabla f_i(\bm x)\}$. In fact, $\partial f_i(\bm x)$ does not contain any useful information of the nearby nonsmooth region, and consequently the search direction ${\bar{\bm{d}}}$ given by \eqref{d_bar} may not be an appropriate descent direction. More precisely, when $\bm x$ is near a point where $f_i$ is not differentiable, this search direction fulfills decrease condition \eqref{dscent-direction} only for very small values of $t$. In contrast, $\partial_\varepsilon f_i(\bm x)$ can capture some local information of the nearby nonsmooth region, and one can stabilize our choice of search direction by replacing $\partial f_i(\bm x)$ by $\partial_\varepsilon f_i(\bm x)$. In this respect, the following minimization problem is considered:
\begin{equation}\label{epsilon-steepest-descent}
\min \big\{ \lVert \boldsymbol{\xi}\rVert \,\, : \,\,  \boldsymbol{\xi}\in\texttt{conv}\{ \cup_{i=1}^p\partial_\varepsilon f_i(\bm x)\} \big\}.
\end{equation}
One major difficulty with problem \eqref{epsilon-steepest-descent} is that, for each objective $f_i$, we need to know the entire subdifferential $ \partial_\varepsilon f_i(\bm x)$, which is not an easy task in many practical situations.  In this regard, we develop an iterative algorithm to efficiently  approximate $\partial_\varepsilon f_i(\bm x)$, for $i=1,\ldots,p$.

For given  $\varepsilon>0, m\in\mathbb{N}$ and $\bm x\in\mathbb{R}^n$, let
$$\bm{\mathcal{S}}_\varepsilon(\bm x):=\left[\begin{array}{cccc}
	\boldsymbol{\xi}_{1,1} & \boldsymbol{\xi}_{1,2} & \ldots & \boldsymbol{\xi}_{1,p} \\
	\boldsymbol{\xi}_{2,1} & \boldsymbol{\xi}_{2,2} & \ldots & \boldsymbol{\xi}_{2,p}\\
	\vdots & \vdots & \ddots& \vdots\\
	\boldsymbol{\xi}_{m,1} & \boldsymbol{\xi}_{m,2} & \ldots & \boldsymbol{\xi}_{m,p}
\end{array} \right], $$ 
be a collection of subgradients computed throughout iterations $1$ to $m$ having the property that
$$\boldsymbol{\xi}_{j,i}\in\partial_\varepsilon f_i(\bm x), \quad \text{for all} \,\, i=1,\ldots,p, \, \, \text{and} \,\, j=1,\ldots,m.$$
In fact, the $i$-th column of $\bm{\mathcal{S}}_\varepsilon(\bm x)$ stores a number of elements of $\partial_\varepsilon f_i(\bm x)$. We denote the $i$-th column of $\bm{\mathcal{S}}_\varepsilon(\bm x)$ by $\bm{\mathcal{S}}_{\varepsilon,i}(\bm x)$. Then, for each $i\in\{1,\ldots,p\}$,
$$\texttt{conv}\{\bm{\mathcal{S}}_{\varepsilon,i}(\bm x) \} \subset \partial_\varepsilon f_i(\bm x), $$
is an inner approximation of $\partial_\varepsilon f_i(\bm x)$, and consequently a practical variant of problem \eqref{epsilon-steepest-descent} can be given by
\begin{equation}\label{Main-subproblem}
\min \big\{ \lVert \boldsymbol{\xi}\rVert \,\, : \,\,  \boldsymbol{\xi}\in\texttt{conv}\{ \cup_{i=1}^p\bm{\mathcal{S}}_{\varepsilon,i}(\bm x)\} \big\}.
\end{equation}
Suppose $\boldsymbol{\xi}^*\neq \bm 0$ solves problem \eqref{Main-subproblem}, and define the normalized direction ${\bm d}\in\mathbb{R}^n$ by
$${\bm d}:=-\boldsymbol{\xi}^*/\lVert \boldsymbol{\xi}^* \rVert.$$
If $\texttt{conv}\{\cup_{i=1}^p\bm{\mathcal{S}}_{\varepsilon,i}(\bm x)\}$ is an adequate approximation of $\texttt{conv}\{ \cup_{i=1}^p\partial_\varepsilon f_i(\bm x)\}$, then one can find the step size $t>0$ satisfying
\begin{equation}\label{Suff-decrease}
f_i(\bm x+t{\bm d})-f_i(\bm x)\leq -\beta t \lVert \boldsymbol{\xi}^* \rVert, \quad \text{for all} \,\, i=1,\ldots,p, \quad \text{and} \quad t\geq \bar t,
\end{equation}
in which $\beta\in(0,1)$ is the Armijo parameter which guarantees a sufficient decrease in objectives $f_i$, and the lower bound $\bar t\in(0, \varepsilon)$ excludes small values for the step size $t$. In the case that condition \eqref{Suff-decrease} is met, we employ the search direction $\bm d$ for taking a serious step, i.e., the current point $\bm x$ is updated by $\bm x^+:=\bm x+t\bm d.$
 Otherwise, we infer that  our approximation of $\texttt{conv}\{ \cup_{i=1}^p\partial_\varepsilon f_i(\bm x)\}$ must be improved. For this purpose, let
 \begin{equation}\label{Insuficient_index}
 	\mathcal{I}:=\left\{i\in\{1,\ldots,p\} \,\, : \,\, \nexists \,\, t\geq \bar t \,\,\,\, \text{\rm s.t.} \,\,\,\, f_i(\bm x+t{\bm d})-f_i(\bm x)\leq -\beta t \lVert \boldsymbol{\xi}^* \rVert  \right\}.
 \end{equation}
 Then, for each $i\in\mathcal{I}$, we need to improve our approximation of $\partial_\varepsilon f_i(\bm x)$. However, in practice, computing the index set $\mathcal{I}$  is not possible because, for some $i$, checking the condition 
 $$\nexists \,\, t\geq \bar t \,\,\,\, \text{\rm s.t.} \,\,\,\, f_i(\bm x+t{\bm d})-f_i(\bm x)\leq -\beta t \lVert \boldsymbol{\xi}^* \rVert,$$
 is impractical. Therefore, in Algorithm \ref{Alg1}, we propose a limited backtracking line search to provide an estimation of the index set $\mathcal{I}$, namely $\tilde{\mathcal{I}}$. Regarding this algorithm, some explanations are as follows. The algorithm consists of two loops. The upper loop tests the sufficient decrease condition \eqref{Suff-decrease} for a limited number of step sizes. If the algorithm terminates within this loop, we conclude $\mathcal{I}=\emptyset$, and hence $\tilde{\mathcal{I}}=\emptyset$ is returned. Moreover, one can employ the obtained step size $t\geq \bar t>0$ to take a serious step. Otherwise, the algorithm executes the lower loop, in which the step size $\bar t>0$ is used to provide the nonempty estimation $\tilde{\mathcal{I}}$ of $\mathcal{I}$. Furthermore, the step size $t$ is set to zero as an indication of taking a null step. Consequently, at termination of Algorithm~\ref{Alg1}, one of the following cases occurs:
 \begin{itemize}
 	\item Case I. Algorithm \ref{Alg1} returns $\{\tilde{\mathcal{I}}, t\}$ with $\tilde{\mathcal{I}}=\emptyset$ and $t\geq\bar t>0$.
 	\item Case II. Algorithm \ref{Alg1} returns $\{\tilde{\mathcal{I}}, t\}$ with $\tilde{\mathcal{I}}\neq\emptyset$ and $t=0$.
 \end{itemize}

 \begin{algorithm}
 	\caption{Limited Backtracking Line Search (LBLS)}
 	\label{Alg1}
 	\hspace*{\algorithmicindent}\textbf{Inputs:} Objectives $\{f_i\}_{i=1}^p$, current point $\bm x\in\mathbb{R}^n$, search direction ${\bm d}=-\boldsymbol{\xi}^*/\lVert \boldsymbol{\xi}^*\rVert$ with $\boldsymbol{\xi}^*\neq \bm 0$, Armijo coefficient  $\beta\in(0,1)$, and lower bound $\bar t\in(0, \varepsilon)$.  \\
 	\hspace*{\algorithmicindent}\textbf{Parameters:}  Reduction factor $r\in(0,1)$, initial step size $t_0\in(\bar t,\infty)$, and number of backtracks $\tau\in\mathbb{N}$ with $r^{\tau}t_0>\bar t> r^{\tau+1}t_0$. \\
 	\hspace*{\algorithmicindent}\textbf{Outputs:} Step size $t\geq 0$ and index set $\tilde{\mathcal{I}}\subset\{1,\ldots,p\}$.\\ 
 	
 	\hspace*{\algorithmicindent}\textbf{Function:} $\{t,\tilde{\mathcal{I}} \}$\,\,=\,\,\texttt{LBLS}\,($\{f_i\}_{i=1}^p, \bm x, \bm d, \beta, \bar t$\,) 
 	
 	\vspace{1mm}
 	\begin{algorithmic}[1]
 		\STATE{\textbf{Initialization:} Set $\tilde{\mathcal{I}}:=\emptyset$ ;}\\
 		\FOR{$t\in\{t_0, rt_0,r^2t_0,\ldots,r^\tau t_0, \bar t \,  \}$}{
 			\IF{$f_i(\bm x+t{\bm d})-f_i(\bm x)\leq -\beta t \lVert \boldsymbol{\xi}^* \rVert$, for all $i=1,\ldots,p$,} \vspace{1mm}
 			\RETURN{\{$t,\tilde{\mathcal{I}}$\} and \textbf{STOP} ;   }
 			\ENDIF
 			
 		}
 		\ENDFOR
 		\FOR{$i\in\{1,\ldots,p\}$} {
 			\IF{$f_i(\bm x+\bar t{\bm d})-f_i(\bm x)> -\beta \bar t\, \lVert \boldsymbol{\xi}^* \rVert$,} \vspace{1mm}
 			\STATE{$\tilde{\mathcal{I}}:=\tilde{\mathcal{I}}\cup\{i\};$}
 			
 			\ENDIF
 		}
 		\ENDFOR
 		\STATE{Set $t:=0$ ;}
 		\RETURN{\{$t,\tilde{\mathcal{I}}$\}  and \textbf{STOP} ;   }
 	\end{algorithmic}
 		\hspace*{\algorithmicindent}\textbf{ End Function}

 \end{algorithm}
 
 If Case I takes place, as already mentioned, we take a serious step by setting $\bm x^+:=\bm x+t \bm d$. Thus, let us proceed with Case II. In this case, we append the following  row of subgradients 
 $$[\boldsymbol{\xi}_{m+1,1}, \boldsymbol{\xi}_{m+1,2}, \ldots, \boldsymbol{\xi}_{m+1,p}],$$
 to the current collection $\bm{\mathcal{S}}_{\varepsilon}(\bm x)$ such that
 \begin{align}\label{Criterions}
 \begin{split}
 &(\text{i})\,\,  \boldsymbol{\xi}_{m+1,i}:=\boldsymbol{\xi}_{m,i}, \,\, \text{for all}\,\, i\notin \tilde{\mathcal{I}}. \\&
 (\text{ii})\,\,  \boldsymbol{\xi}_{m+1,i}\in\partial_\varepsilon f_i(\bm x)\,\, \text{and} \,\, 
 \boldsymbol{\xi}_{m+1,i}\notin\texttt{conv} \{\mathcal{S}_{\varepsilon,i}(\bm x)\},\,\,   \text{for all} \,\, i\in \tilde{\mathcal{I}}.
 \end{split}
   \end{align}
 The first part of condition \eqref{Criterions}  states that for the objectives $f_i$ with $i\notin \tilde{\mathcal{I}}$, we do not compute a new element of $\partial_\varepsilon f_i(\bm x)$. By the second part of this condition, for each $i\in\tilde{\mathcal{I}}$, we compute a new element of  $\partial_\varepsilon f_i(\bm x)$ which significantly improves our approximation of  $\partial_\varepsilon f_i(\bm x)$.
In other words, if we update $\bm{\mathcal{S}}_{\varepsilon, i}(\bm x)$ by 
 $$\bm{\mathcal{S}}_{\varepsilon, i}^+(\bm x):=\bm{\mathcal{S}}_{\varepsilon, i}(\bm x)\cup\{\boldsymbol{\xi}_{m+1, i} \},$$
 then part (ii) implies $\texttt{conv}\{\bm{\mathcal{S}}_{\varepsilon,i}(\bm x) \}\subsetneq \texttt{conv}\{\bm{\mathcal{S}}_{\varepsilon,i}^+(\bm x) \}$.  Concerning part (ii) of condition \eqref{Criterions}, in the following lemma, we present a sufficient condition for the new subgradient $\boldsymbol{\xi}_{m+1,i}\in\partial_\varepsilon f_i(\bm x)$ to  satisfy $$\boldsymbol{\xi}_{m+1,i}\notin\texttt{conv} \{\mathcal{S}_{\varepsilon,i}(\bm x)\}.$$ 
\begin{lemma}
	Assume $ \bar i\in\tilde{\mathcal{I}}$ and $\boldsymbol{\xi}^*\neq \bm 0$ solves problem \eqref{Main-subproblem}. For the direction ${\bm d}=-\boldsymbol{\xi}^*/\lVert \boldsymbol{\xi}^* \rVert$ and scalar $c\in(0,1)$, suppose  $\boldsymbol{\xi}_{m+1,\bar i}\in\partial_\varepsilon f_{\bar i}(\bm x)$ satisfies
	\begin{equation}\label{Criterion1}
	 \boldsymbol{\xi}_{m+1,\bar i}^T {\bm d} \geq -c \lVert  \boldsymbol{\xi}^*\rVert.
	\end{equation}
	Then $\boldsymbol{\xi}_{m+1,\bar i}\notin\text{\rm\texttt{conv}}\{\bm{\mathcal{S}}_{\varepsilon, \bar i}(\bm x)\}$.
\end{lemma}
\begin{proof}
	Since $\boldsymbol{\xi}^*$ is the optimal solution of problem \eqref{Main-subproblem}, one can conclude \cite{Rockafellar2004}
	$$-\boldsymbol{\xi}^*\in \mathcal{N}\left(\texttt{conv} \{\cup_{i=1}^{p}\mathcal{S}_{\varepsilon,i}(\bm x) \}, \boldsymbol{\xi}^*\right), $$
	which implies 
	\begin{equation}\label{L1-1}
	\boldsymbol{\xi}^T \boldsymbol{\xi}^*\geq \lVert \boldsymbol{\xi}^*\rVert^2, \quad \text{for all} \,\, \boldsymbol{\xi}\in \texttt{conv} \{\cup_{i=1}^{p}\mathcal{S}_{\varepsilon,i}(\bm x) \}.
	\end{equation}
	Thus, if $\boldsymbol{\xi}_{m+1, \bar i}\in\partial_\varepsilon f_{\bar i}(\bm x)$ satisfies $\boldsymbol{\xi}_{m+1, \bar i}^T {\bm d} \geq -c \lVert \boldsymbol{\xi}^*\rVert$, inequality \eqref{L1-1} yields $\boldsymbol{\xi}_{m+1, \bar i}\notin\texttt{conv} \{\cup_{i=1}^{p}\mathcal{S}_{\varepsilon,i}(\bm x) \}$, and consequently 
	$$\boldsymbol{\xi}_{m+1, \bar i}\notin \text{\rm\texttt{conv}}\{\bm{\mathcal{S}}_{\varepsilon, \bar i}(\bm x)\}. $$
	 \end{proof}
	 In Algorithm \ref{Line-Search}, we develop a new variant of Mifflin's line search \cite{kiwielbook} to find a new subgradient for objective $f_i,\, i\in\tilde{\mathcal{I}}$, satisfying condition \eqref{Criterion1}. This algorithm consists of two conditional blocks. The first one modifies the interval within which we look for an effective subgradient, and the second one checks condition~\eqref{Criterion1}. It is also noted that, in this algorithm, $t_s$ varies within the interval $(0, \varepsilon)$ and $\boldsymbol{\xi}_s\in\partial f_i(\bm x + t_s \bm d)$, for every $s\geq 0$. Therefore, the output of this algorithm is an element of $\partial_\varepsilon f_i(\bm x)$.
	 
	 \begin{algorithm}
	 	\caption{Finding an Effective Subgradient (FES)}
	 	\label{Line-Search}
	 	\hspace*{\algorithmicindent}\textbf{Inputs:} Objective $f_i$ with $i\in\tilde{\mathcal{I}}$, current point $\bm x\in\mathbb{R}^n$, search direction ${\bm d}=-\boldsymbol{\xi}^*/\lVert \boldsymbol{\xi}^*\rVert$ with $\boldsymbol{\xi}^*\neq \bm 0$, and radius $\varepsilon\in(0,1)$. \\
	 	\hspace*{\algorithmicindent}\textbf{Parameters:} Armijo coefficient $\beta\in(0,1)$ and lower bound $\bar t\in (0,\varepsilon)$  as set in Algorithm \ref{Alg1}, scale parameter $c\in(0,1)$ with $\beta<c$, and reduction factor $\eta\in(0,0.5)$. \\
	 	\hspace*{\algorithmicindent}\textbf{Output:} A subgradient $\boldsymbol{\xi}\in\partial_\varepsilon f_{i}(\bm x)$.\\
	 	
	 	\hspace*{\algorithmicindent}\textbf{Function:} $\boldsymbol{\xi}$\,\,=\,\,\texttt{FES}\,($f_i$, $\mathbf x$, ${\bm d}$, $\varepsilon$)
	 	
	 	\begin{algorithmic}[1]
	 		\STATE{\textbf{Initialization:} Set $t_0:=\bar t$, compute $\boldsymbol{\xi}_0\in\partial f_i(\bm x+t_0{\bm d})$, and set $t^l_0:=0, t^u_0:=\varepsilon$, $s:=0$ ;}
	 		\WHILE{ true }
	 		
	 		\IF{$f_i(\bm x+t_s{\bm d})-f_i(\bm x)\leq -\beta\, t_s\, \lVert \boldsymbol{\xi}^*\rVert,$}
	 		\STATE {Set $t^l_{s+1}:=t_s, \,\,\,\, t^u_{s+1}:=t^u_s$ ;} 
	 		\ELSE 
	 		\STATE{Set $t^l_{s+1}:=t^l_s, \,\,\,\, t^u_{s+1}:=t_s$ ;}
	 		\ENDIF
	 		
	 	   \IF{$\boldsymbol{\xi}_s^T{\bm d}\geq -c \lVert \boldsymbol{\xi}^*\rVert ,$} 
	 		
	 		\RETURN {$\boldsymbol{\xi}_s$ and \textbf{STOP} ;}
	 		\ENDIF
	 		\STATE{Choose $t_{s+1}\in \left[ t^l_{s+1}+\eta(t^u_{s+1}-t^l_{s+1}),\, t^u_{s+1}-\eta(t^u_{s+1}-t^l_{s+1})   \right]$  ;} \label{Line11}
	 		\STATE{Compute $\boldsymbol{\xi}_{s+1}\in \partial f_i(\bm x+t_{s+1} {\bm d})$ ;}
	 		\STATE{Set $s:=s+1$ ;}
	 		
	 		\ENDWHILE
	 	\end{algorithmic}
	 	\hspace*{\algorithmicindent}\textbf{ End Function}
	 \end{algorithm}

In the rest of this section, it is proved that Algorithm~\ref{Line-Search} terminates after finitely many iterations. To this end, we provide some auxiliary results in the next lemma.

\begin{lemma}\label{L2} If  Algorithm \ref{Line-Search} does not terminate ($s\to\infty$), then the followings hold:
	\begin{itemize}
		\item[(i)]  We have $t_s\in\{t_{s+1}^l, t_{s+1}^u  \}$, for all $s\geq 0$. In addition, for all $s\geq 1$,
		\begin{align}
		&0<t^u_{s+1}-t^l_{s+1}\leq (1-\eta) (t^u_{s}-t^l_{s}), \label{L2-1}\\&
		0\leq t^l_s\leq t^l_{s+1}< t^u_{s+1}\leq t^u_s\leq \varepsilon. \label{L2-2}
		\end{align}
		\item[(ii)] There exists $t^*\in[0,\varepsilon]$ such that $t^u_s\downarrow t^*$, $t^l_s\uparrow t^*$, and $t_s\to t^*$ as $s\to \infty$. Moreover $$t^*\in \bm{T}:=\{t\,\, : \,\, f_i(\bm x +t {\bm d})-f_i(\bm x)\leq -\beta\, t\,  \lVert\boldsymbol{\xi}^*\rVert   \}.$$
		\item[(iii)] Suppose $\bm{S}:=\{ s\,\, : \, \, t^u_{s+1}=t_s \}$. Then $\bm{S}$ is an infinite set. 
		
	\end{itemize}

\end{lemma}
\begin{proof}
	(i) This part follows immediately from the construction of the algorithm.
	
	(ii) It follows from \eqref{L2-1} that the sequence $\{t^u_s-t^l_s\}_{s}$ is decreasing and bounded by zero. Therefore, there exists $a\in\mathbb{R}$ such that $(t^u_s-t^l_s)\to a$ as $s\to\infty$. Letting $s$ approach infinity in inequality \eqref{L2-1}, one can conclude $0\leq a\leq(1-\eta)a$. Since $\eta\in(0,0.5)$, we deduce $a=0$. As $(t^u_s-t^l_s)\to 0$ when $s\to\infty$, inequalities~ \eqref{L2-2} ensure the existence of $t^*\in[0, \varepsilon]$ such that $t^l_s\uparrow t^*, t^u_s\downarrow t^*$ as $s\to\infty$. Moreover, since $t_s\in\{ t^l_{s+1}, t^u_{s+1} \}$ for all $s\geq 0$, we have $t_s\to t^*$ as $s\to\infty$.  It remains to show that $t^*\in\bm T$. We have  $t^l_s\in \bm T$ for all $s\geq 0$, i.e.,
	$$f_i(\bm x+t^l_s {\bm d})-f_i(\bm x)\leq -\beta\, t^l_s\,\lVert\boldsymbol{\xi}^*\rVert, \qquad \text{for all} \,\, s\geq 0. $$
	 Now, continuity of $f_i$ along with the fact that $t^l_s\uparrow t^*$ as $s\to\infty$  yields  
	$$f_i(\bm x+t^* {\bm d})-f_i(\bm x)\leq -\beta\, t^*\,\lVert\boldsymbol{\xi}^*\rVert, $$
	and therefore $t^*\in \bm T.$

	(iii) First, we show  $\bm S\neq\emptyset$. Assume by contradiction that $\bm{S}=\emptyset$, which means
	\begin{equation}\label{L2-3}
	 f_i(\bm x+t_s {\bm d})-f_i(\bm x)\leq -\beta\, t_s\,  \lVert\boldsymbol{\xi}^*\rVert, \qquad \text{for all} \,\, s\geq 0.
	\end{equation}
	In particular, for $s=0$, we have $t_0=\bar t$, and thus
	$$f_i(\bm x+\bar t {\bm d})-f_i(\bm x)\leq -\beta\, \bar t\,  \lVert\boldsymbol{\xi}^*\rVert, $$
	yielding $i\notin \tilde{\mathcal{I}}$ (see Algorithm \ref{Alg1}), which is a contradiction with the fact that $i\in\tilde{\mathcal{I}}$ in Algorithm \ref{Line-Search}.
	  Thus, $\bm S\neq \emptyset$. Next, we prove that $\bm S$ is indeed an infinite set. Suppose by contradiction that $\bm S$ is finite. Then, since $\bm {S}\neq\emptyset$ and $t^u_s\downarrow t^*$ as $s\to\infty$, one can find $\bar s\in\mathbb{N}$ such that 
	$$t^u_s=t^*, \qquad \text{for all} \,\, s\geq\bar s \quad \text{and} \quad t^u_s>t^*, \qquad \text{for all} \,\, s<\bar s. $$
	Consequently, $t^*=t^u_{\bar s}=t_{\bar s-1}$, and therefore
	$$f_i(\bm x+t_{\bar{s}-1} {\bm d})-f_i(\bm x)> -\beta\, t_{\bar{s}-1}\,  \lVert\boldsymbol{\xi}^*\rVert, $$
	which gives $t^*\notin \bm T$, contradicting the fact  $t^*\in \bm T$.
\end{proof}	
Now, we are in a position to state the principal convergence result for Algorithm \ref{Line-Search}. 

\begin{theorem}
In Algorithm \ref{Line-Search} 	suppose that objective function $f_i:\mathbb{R}^n\to\mathbb{R}$ is weakly upper semismooth. Then this algorithm  terminates after a finite number of iterations.
\end{theorem}

\begin{proof}\label{T1}
	By indirect proof, assume that Algorithm \ref{Line-Search} does not terminate. Let $\bm {S}$ be as defined in part (iii) of Lemma \ref{L2}. Then $\bm{S}$ is infinite and 
	\begin{equation}\label{T1-1}
	f_i(\bm x+t_s {\bm d})-f_i(\bm x)>-\beta t_s \lVert \boldsymbol{\xi}^*\rVert, \quad \text{for all} \,\, s\in\bm{S}.
	\end{equation}
	By part (ii) of Lemma \ref{L2}, we know that $t_s\to t^*$ as $s\to\infty$ with $t^*\in \bm T$, i.e.,
	\begin{equation}\label{T1-2}
	f_i(\bm x+t^* {\bm d})-f_i(\bm x)\leq-\beta t^* \lVert \boldsymbol{\xi}^*\rVert.
	\end{equation}
	Plugging the latter inequality into \eqref{T1-1}, we obtain
	\begin{equation}\label{T1-3}
	f_i(\bm x+t_s {\bm d})-f_i(\bm x+t^* {\bm d})>-\beta \lVert\boldsymbol{\xi}^*\rVert(t_s-t^*), \quad \text{for all}\,\, s\in\bm{S}.
	\end{equation}
	Define  $\bm z:=\bm x+t^*{\bm d}$ and $h_s:=t_s-t^*>0$, for all $s\in\bm{S}$. Then \eqref{T1-3} can be rewritten as
	\begin{equation}\label{T1-4}
	-\beta\lVert\boldsymbol{\xi}^*\rVert< \frac{f_i(\bm z+h_s {\bm d})-f_i(\bm z)}{h_s}, \quad \text{for all}\,\, s\in\bm{S}.
	\end{equation}
	Since the objective $f_i$ is weakly upper semismooth, inequality \eqref{T1-4} implies
	\begin{equation}\label{T1-5}
	-\beta\lVert\boldsymbol{\xi}^*\rVert\leq \liminf_{s\xrightarrow{\bm{S}}\infty} \frac{f_i(\bm z+h_s {\bm d})-f_i(\bm z)}{h_s}\leq \limsup_{s\xrightarrow{\bm{S}}\infty} \boldsymbol{\xi}_s^T {\bm d}. 
	\end{equation}
	On the other hand, since $s\to\infty$ in Algorithm \ref{Line-Search}, the stopping criterion of this algorithm is never met, and hence   
	$$\boldsymbol{\xi}_s^T {\bm d}<-c\lVert\boldsymbol{\xi}^*\rVert, \quad \text{for all} \,\, s\in\bm{S}, $$
	and consequently
	$$ \limsup_{s\xrightarrow{\bm{S}}\infty} \boldsymbol{\xi}_s^T {\bm d}\leq -c\lVert\boldsymbol{\xi}^*\rVert <-\beta\lVert\boldsymbol{\xi}^*\rVert,$$
	which is a contradiction with \eqref{T1-5}. 
\end{proof}
\section{Finding a Clarke substationary point}\label{Sec4}
The main goal of this section is to employ the provided tools in Section \ref{Sec3} in order to locate a Clarke substationary point for multiobjective minimization problem \eqref{Main-Problem}. For this purpose, the following concept of substationarity plays a crucial role.
\begin{definition}[\cite{bagirov2012,bagirov2010,Multi-sub4}]\label{Def1}
	Suppose $\varepsilon>0$ and $\delta>0$ are given. A point $\bm x\in\mathbb{R}^n$ is called a $\big(\delta, \mathcal{S}_\varepsilon(\bm x)\big)$-substationary point for problem \eqref{Main-Problem} if 
	$$\min\big\{ \lVert \boldsymbol{\xi}\rVert \,\, : \,\, \boldsymbol{\xi}\in\text{\rm \texttt{conv}}\{ \cup_{i=1}^p\bm{\mathcal{S}}_{\varepsilon,i}(\bm x)\}  \big\}\leq \delta.$$
\end{definition}
\subsection{Computing a $\big(\delta, \mathcal{S}_\varepsilon(\bm x)\big)$-substationary point }
Based on the discussions provided in the previous section, a procedure  for finding a $\big(\delta, \mathcal{S}_\varepsilon(\bm x)\big)$-substationary point is presented in Algorithm \ref{Alg3}. Next, we aim to show that Algorithm \ref{Alg3} terminates after finite number of iterations. To this end, set
\begin{equation}\label{Index-set}
\mathcal{K}:=\{k\in\mathbb{N}_0 \,\, : \,\, {t}_k>0 \,\, \text{in Algorithm \ref{Alg3}}  \}.
\end{equation}

\begin{algorithm}
	\caption{Finding a $\big(\delta, \mathcal{S}_\varepsilon(\bm x)\big)$-Substationary Point}
	\label{Alg3}
	\hspace*{\algorithmicindent}\textbf{Inputs:} Objective functions $\{f_i\}_{i=1}^p$, starting point $\bm x_0\in\mathbb{R}^n$, radius $\varepsilon\in(0,1)$, substationarity tolerance $\delta>0$, and Armijo coefficient $\beta\in(0,1)$.  \\
	\hspace*{\algorithmicindent}\textbf{Parameter:} Lower bound $\bar t\in(0, \varepsilon).$ \\
	\hspace*{\algorithmicindent}\textbf{Output:} A $(\delta, \mathcal{S}_\varepsilon(\bm x))$-substationary point $\bm x\in\mathbb{R}^n$.\\
	
	\hspace*{\algorithmicindent}\textbf{Function:} $\bm x$\,\,=\,\,\texttt{DS-SP}\,($\{f_i\}_{i=1}^p$, $\bm x_0, \varepsilon, \delta$, $\beta$) 
	
	\begin{algorithmic}[1]
		\STATE{\textbf{Initialization:} Compute $\boldsymbol{\xi}_{0,i}\in\partial f_i(\bm x_0)$, and set  $\mathcal{S}^0_{\varepsilon,i}(\bm x_0):=\{\boldsymbol{\xi}_{0,i}\}$, for all $i\in\{1,\ldots,p\}$. Set $k:=0$ ;}
		\WHILE{ true }
		\STATE{Set $\boldsymbol{\xi}^*_k:=\arg\min\big\{\lVert \boldsymbol{\xi}\lVert \,\, : \,\, \boldsymbol{\xi}\in\text{\rm \texttt{conv}}\{ \cup_{i=1}^p\bm{\mathcal{S}}^k_{\varepsilon,i}(\bm x_k)\}   \big\}$ ;}
		\IF{$\lVert \boldsymbol{\xi}^*_k\rVert\leq \delta,$}
		\RETURN {$\bm x_k$ as a $(\delta, \mathcal{S}_\varepsilon(\bm x_k))$-substationary point and \textbf{STOP} ; }
		\ENDIF
		\STATE{Set  $\bm d_k:=-\boldsymbol{\xi}^*_k/\lVert\boldsymbol{\xi}^*_k\rVert$ ; }
		\STATE{Set $\{t_k, \tilde{\mathcal{I}_k} \}:=$ \texttt{LBLS} ($\{f_i\}_{i=1}^p, \bm x_k, \bm d_k, \beta, \bar t$\,)  ;}
		
		\IF{$t_k>0,$}
		\STATE {Set $\bm x_{k+1}:=\bm x_k+ t_k \bm d_k$ ;}
		\STATE{Compute $\boldsymbol{\xi}_{k+1,i}\in\partial f_i(\bm x_{k+1})$, for all $i\in\{1,\ldots,p\}$ ; }
		\STATE{Set  $\mathcal{S}^{k+1}_{\varepsilon,i}(\bm x_{k+1}):=\{\boldsymbol{\xi}_{k+1,i}\}$, for all $i\in\{1,\ldots,p\}$ ; }
		\ENDIF
		\IF{$t_k=0,$} 
		\STATE { Set $\boldsymbol{\xi}_{k+1, i}:=\texttt{FES}(f_i, \bm x_{k}, \bm d_k, \varepsilon)$, for all $i\in\tilde{\mathcal{I}}_k$  ;}
		\STATE { Set $\boldsymbol{\xi}_{k+1, i}:=\boldsymbol{\xi}_{k,i}$, for all $i\notin\tilde{\mathcal{I}}_k$  ;}
		\STATE{Set $\bm x_{k+1}:=\bm x_k$ ;}   
		\STATE{Set $\mathcal{S}^{k+1}_{\varepsilon, i}(\bm x_{k+1}):= \mathcal{S}^k_{\varepsilon, i}(\bm x_{k})\cup\{\boldsymbol{\xi}_{k+1,i} \}$, for all $i\in\{1,\ldots,p\}$   ;}
		\ENDIF
		\STATE{Set $k:=k+1$ ;}
		
		\ENDWHILE
	\end{algorithmic}
	\hspace*{\algorithmicindent}\textbf{ End Function}
\end{algorithm}
We also make the following assumption about the multiobjective problem~\eqref{Main-Problem}.
\begin{assumption}\label{Assumption 1}
	Objective functions $\{f_i\}_{i=1}^p$ are weakly upper semismooth. Moreover, for some $j\in\{1,\ldots,p\}$, the sublevel set
	$$lev_{f_{j}(\bm x_0)}(f_{j}):=\{ \bm x\in\mathbb{R}^n \,\, : \,\, f_{j}(\bm x)\leq f_{j}({\bm x_0})   \},$$
	is bounded.
\end{assumption}
We start with the following lemma. 
\begin{lemma}\label{L3}
	Suppose that Assumption \ref{Assumption 1} holds. If Algorithm \ref{Alg3} does not terminate, then the index set $\mathcal{K}$ is finite.
\end{lemma}
\begin{proof}
	Since $f_{j}$ is continuous,  $lev_{f_{j}(\bm x_0)}(f_{j})$ is closed, and hence, by assumption, $lev_{f_{j}(\bm x_0)}(f_{j})$ is compact. Therefore 
	\begin{equation}\label{L3-0}
	f^*_{j}:=\min\{f_{j}(\bm x) \,\, : \,\, \bm x\in\mathbb{R}^n \}> -\infty.
	\end{equation}
	By indirect proof, suppose that $\mathcal K$ is infinite. As Algorithm \ref{Alg3} does not terminate, we conclude
	\begin{equation}\label{L3-1}
	\lVert \boldsymbol{\xi}^*_k\rVert>\delta, \quad \text{for all}\,\, k.
	\end{equation}
	 For each $k\in\mathcal{K}$, we have ${t}_k>0$ and thus
	\begin{equation*}
	f_i(\bm x_{k+1})-f_i(\bm x_k)\leq -\beta\,  t_{k} \,\lVert\boldsymbol{\xi}^*_k\rVert, \quad \text{for all} \,\, i=1,\ldots,p \,\,\, \text{and} \,\,\, k\in\mathcal{K}. 
	\end{equation*}
	By construction of Algorithm \ref{Alg1}, we have $t_{k}\geq \bar t>0$, for all $k\in\mathcal{K}$. Thus, in virtue of \eqref{L3-1}, the latter inequality implies
	\begin{equation*}
	f_i(\bm x_{k+1})-f_i(\bm x_k)\leq -\beta\, \bar t\, \delta, \quad \text{for all} \,\,i=1,\ldots,p \,\,\, \text{and} \,\,\, k\in\mathcal{K}.
	\end{equation*}
	Furthermore, for any $k\in\mathbb{N}_0\setminus\mathcal{K}$, we have ${t_k}=0$ and hence
	\begin{equation*}
	f_i(\bm x_{k+1})=f_i(\bm x_k), \quad \text{for all} \,\, i=1,\ldots,p \,\,\, \text{and} \,\,\, k\in\mathbb{N}_0\setminus\mathcal K.
	\end{equation*}
	In particular, for $i=j$, one can write
	\begin{align}
	& f_j(\bm x_{k+1})-f_j(\bm x_k)\leq -\beta\, \bar t\, \delta, \quad \text{for all} \,\, k\in\mathcal{K}, \label{L3-2} \\&
	f_j(\bm x_{k+1})=f_j(\bm x_k), \quad \text{for all} \,\, k\in\mathbb{N}_0\setminus\mathcal K. \label{L3-3}
	\end{align}
	Using \eqref{L3-2} and \eqref{L3-3} inductively, we arrive at
	\begin{equation}\label{L3-4}
	f_j(\bm x_{k+1})\leq f_j(\bm x_0)-\sum_{\substack{q\in\mathcal{K}\\ q\leq k+1}} \beta\, \bar t \,\delta.
	\end{equation}
	As $\mathcal{K}$ is infinite, $\sum_{\substack{q\in\mathcal{K}\\ q\leq k+1}} \beta\, \bar t\, \delta\to\infty$ as $k\to\infty$. Thus, \eqref{L3-4} gives $f_j(\bm x_k)\to-\infty$ as $k\to\infty$, which contradicts \eqref{L3-0}.
\end{proof}
Our  main result about Algorithm \ref{Alg3} is given in the next theorem.  

\begin{theorem}\label{T2}
	Suppose that Assumption \ref{Assumption 1} holds. Then Algorithm \ref{Alg3} terminates in finite number of iterations.
\end{theorem}
\begin{proof}
By indirect proof, suppose that Algorithm \ref{Alg3} does not terminate, i.e., $k\to\infty$. Therefore
\begin{equation}\label{T2-0}
\lVert \boldsymbol{\xi}_k^*\lVert>\delta, \quad \text{for all} \,\, k.
\end{equation}
	Let $\mathcal{K}$ be as defined in \eqref{Index-set}. Then, in view of Lemma \ref{L3}, $\mathcal{K}$ is a finite set. In the case that $\mathcal{K}\neq\emptyset$, set
	$$\bar{k}:=\max\{k \,\, : \,\, k\in\mathcal{K}   \}.$$
	Otherwise, we set $\bar{k}:=0$. Define $\bar{\bm x}:=\bm x_{\bar k+1}$. Then, for any $k>\bar k$, we have ${t}_k=0$, and consequently
	  $$\bar{\bm x}=\bm x_{k}, \quad \text{for all}\,\, k>\bar k.$$ 
	  Furthermore, for any $k>\bar k$, the following collection of new subgradients is computed:
	  $$\boldsymbol{\chi}_k:=\{\boldsymbol{\xi}_{k+1,i} \,\, : \,\, \boldsymbol{\xi}_{k+1,i}=\texttt{FES}(f_i, \bar{\bm x}, \bm d_k, \varepsilon)\in\partial_\varepsilon f_i(\bar{\bm x}), \,\, i\in\tilde{\mathcal{I}}_k  \}.$$
	  By construction of Algorithm \ref{Line-Search}, each element of $\boldsymbol{\chi}_k$ satisfies 
\begin{equation*}
\boldsymbol{\xi}_{k+1,i}^T{\bm d}_k\geq -c\lVert \boldsymbol{\xi}^*_k\lVert.
\end{equation*}
Since  ${\bm d}_k=-\boldsymbol{\xi}^*_k/\lVert \boldsymbol{\xi}^*_k\lVert$, the above inequality is represented as
\begin{equation}\label{T2-2}
\boldsymbol{\xi}_{k+1,i}^T\boldsymbol{\xi}^*_k\leq c\lVert \boldsymbol{\xi}^*_k\lVert^2.
\end{equation}
For all $k>\bar k$, we have $\tilde{\mathcal{I}}_k\neq\emptyset$.  Let $i_k\in\tilde{\mathcal{I}}_k$ be arbitrary. Then,
{$\boldsymbol{\xi}_{k+1,i_k}\in\boldsymbol{\chi}_k$}, and hence
\begin{equation}\label{T2-3}
\boldsymbol{\xi}_{k+1,i_k}\in \mathcal{S}^{k+1}_{\varepsilon, i_k}(\bm x_{k+1})=\mathcal{S}^{k+1}_{\varepsilon, i_k}(\bar{\bm x})\subset\partial_\varepsilon f_{i_k}(\bar{\bm x})\subset\texttt{conv}\{\cup_{i=1}^p\partial_\varepsilon f_i(\bar{\bm x})\}.
\end{equation}
It is  also noted that, for all $k>\bar k$,
\begin{equation}\label{T2-3'}
\boldsymbol{\xi}^*_k\in\texttt{conv}\{\cup_{i=1}^p \mathcal{S}^k_{\varepsilon,i}({\bm x_k}) \}=\texttt{conv}\{\cup_{i=1}^p \mathcal{S}^k_{\varepsilon,i}(\bar{\bm x}) \}\subset \texttt{conv}\{\cup_{i=1}^p\partial_\varepsilon f_i(\bar{\bm x})\}.
\end{equation}
  Boundedness of $\texttt{conv}\{\cup_{i=1}^p\partial_\varepsilon f_i(\bar{\bm x})\}$  implies
$$C:=\sup\{ \lVert\boldsymbol{\xi}\rVert\,\,:\,\, \boldsymbol{\xi}\in\texttt{conv}\{\cup_{i=1}^p\partial_\varepsilon f_i(\bar{\bm x})\} \}<\infty.$$
Set $\bar C:=\max\{C, \delta\}$. Therefore, it follows from \eqref{T2-3} and \eqref{T2-3'} that
\begin{equation}\label{T2-4}
\lVert \boldsymbol{\xi}_{k+1, i_k}-\boldsymbol{\xi}^*_k\rVert \leq 2\bar C, \quad \text{for all} \,\, k>\bar k.
\end{equation}
Now, for any $t\in(0,1)$ and $k>\bar k$, one can write
\begin{align}\label{T2-5}
\lVert \boldsymbol{\xi}^*_{k+1}\rVert^2&\leq \lVert t \boldsymbol{\xi}_{k+1,i_k} + (1-t) \boldsymbol{\xi}^*_k\lVert^2\nonumber \\&
=t^2\lVert \boldsymbol{\xi}_{k+1,i_k}-\boldsymbol{\xi}^*_k\lVert^2+2t(\boldsymbol{\xi}^*_k)^T (\boldsymbol{\xi}_{k+1,i_k}-\boldsymbol{\xi}^*_k)+\lVert \boldsymbol{\xi}^*_k\rVert^2.
\end{align}
In view of \eqref{T2-2} and \eqref{T2-4}, one can continue \eqref{T2-5} as
\begin{align}
\lVert  \boldsymbol{\xi}^*_{k+1}\rVert^2&\leq 4t^2\bar C^2+2tc\lVert \boldsymbol{\xi}^*_k\lVert^2-2t \lVert \boldsymbol{\xi}^*_k\lVert^2+ \lVert \boldsymbol{\xi}^*_k\lVert^2 \nonumber\\&
= 4t^2\bar C^2+ \left(1-2t(1-c)\right) \lVert \boldsymbol{\xi}^*_k\lVert^2 \nonumber \\&
=:\varphi(t), \label{T2-6}
\end{align}
for all $t\in(0,1)$. It is easy to see that
 $$t^*:=2(1-c)\lVert \boldsymbol{\xi}^*_k\lVert^2/8\bar C^2\in(0,1),$$
  is the minimum point of  $\varphi(t)$ and 
\begin{equation*}
\varphi(t^*)=\left(1-\frac{(1-c)^2\lVert \boldsymbol{\xi}^*_k\lVert^2}{4\bar C^2} \right) \lVert \boldsymbol{\xi}^*_k\lVert^2.
\end{equation*}
In virtue of \eqref{T2-0}, the above equality gives
\begin{equation}\label{T2-6'}
\varphi(t^*)\leq\left(1-\frac{(1-c)^2\delta^2}{4\bar C^2} \right) \lVert \boldsymbol{\xi}^*_k\lVert^2.
\end{equation}
Since $\delta\leq \bar C$ and $c\in(0,1)$, we get 
$$\alpha:=1-\frac{(1-c)^2\delta^2}{4\bar C^2}\in(0,1).$$
 Next, it follows from \eqref{T2-6} and \eqref{T2-6'} that
\begin{equation}\label{T2-7}
0\leq \lVert \boldsymbol{\xi}^*_{k+1}\rVert^2 \leq \varphi(t^*)\leq \alpha \lVert\boldsymbol{\xi}^*_k\lVert^2, \quad \text{for all} \,\, k>\bar k,
\end{equation}
yielding the sequence $\{\lVert \boldsymbol{\xi}^*_{k}\rVert^2 \}_{k>\bar k}$ is bounded from below and decreasing. Thus, it converges. Let $\lVert \boldsymbol{\xi}^*_{k}\rVert^2 \to A$ as $k\to\infty$. Letting $k\to\infty$  in inequality \eqref{T2-6}, we obtain $0\leq A\leq\alpha A$. Since $\alpha\in(0,1)$, we deduce $A=0$. Consequently, $\lVert \boldsymbol{\xi}^*_{k}\rVert^2 \to 0$ as $k\to\infty$, which contradicts \eqref{T2-0}.

\end{proof}

\subsection{Computing  a Clarke substationary point}\label{Sec5}
In this section, for the given decreasing sequences  $\delta_\nu\downarrow 0$ and $\varepsilon_\nu\downarrow 0$, we seek a Clarke substationary point of problem \eqref{Main-Problem} by generating a sequence of $\big(\delta_\nu, \mathcal{S}_{\varepsilon_\nu}(\cdot)\big)$-substationary points. Such a sequence is simply generated by Algorithm \ref{Alg4}.
\begin{algorithm}
	\caption{Computation of a Clarke stationary point}
	\label{Alg4}
	\hspace*{\algorithmicindent}\textbf{Inputs:} Objective functions $\{f_i\}_{i=1}^p$, starting point $\bm x_0\in\mathbb{R}^n$, positive sequences $\delta_\nu\downarrow 0$ and $\varepsilon_\nu\downarrow 0$, and optimality tolerance $\rho>0$. \\
	\hspace*{\algorithmicindent}\textbf{Parameter:} Armijo coefficient $\beta>0$. \\
	\hspace*{\algorithmicindent}\textbf{Output:} A point $\bm x\in\mathbb{R}^n$ as an approximation of a Clarke substationary point.\\
	
	
	\begin{algorithmic}[1]
		\STATE{\textbf{Initialization:} Set $\nu:=0$ ;}
		\WHILE{ true }
		\STATE{Set $\bm x_{\nu+1}:=\,\texttt{DS-SP}\,(\{f_i\}_{i=1}^p, \bm x_\nu, \varepsilon_\nu, \delta_\nu, \beta)$ ;}
		\IF{$ \delta_\nu < \rho,$ \rm{\textbf{and}} $ \varepsilon_\nu < \rho,$}
		\RETURN {$\bm x_\nu$ as an approximation of a Clarke substationary point and \textbf{Stop} ; }
		\ENDIF
		\STATE{Set $\nu:=\nu+1$ ;}
		
		\ENDWHILE
	\end{algorithmic}

\end{algorithm}

To study the asymptotic behavior of Algorithm \ref{Alg4}, one may assume $\rho=0$. In this case, the algorithm generates the infinite sequence $\{\bm x_\nu\}_\nu$. In the next theorem, it is proved that any accumulation point of the sequence $\{\bm x_\nu\}_\nu$ is Clarke substationary for problem \eqref{Main-Problem}.

\begin{theorem}\label{T3}
	Suppose that  Assumption \ref{Assumption 1} holds. If $\rho=0$ in Algorithm \ref{Alg4}, then any accumulation point of the sequence $\{\bm x_\nu\}_\nu$ generated by this algorithm is Clarke substationary for problem \eqref{Main-Problem}. 
\end{theorem}
\begin{proof}
During each iteration $\nu\geq 0$ of Algorithm \ref{Alg4}, the \texttt{DS-SP} algorithm is run and terminated after finite number of iterations. For any $\nu\geq 0$, suppose that this algorithm is terminated at iteration $k_\nu\in\mathbb{N}_0$. Thus
$$\bm x_{\nu+1}=\bm x_{k_\nu}=\,\texttt{DS-SP}\,(\{f_i\}_{i=1}^p, \bm x_\nu, \varepsilon_\nu, \delta_\nu, \beta),$$
and therefore
\begin{equation}\label{T3-1}
\min\big\{\lVert \boldsymbol{\xi}\lVert \,\, : \,\, \boldsymbol{\xi}\in\text{\rm \texttt{conv}}\{ \cup_{i=1}^p\bm{\mathcal{S}}^{k_\nu}_{\varepsilon_\nu,i}(\bm x_{\nu+1})\}   \big\}\leq \delta_\nu, \quad \text{for all}\,\, \nu\geq 0. 
\end{equation}
We note that  $\bm x_\nu\in lev_{f_j(\bm x_0)}(f_j)$, for all $\nu\geq 0$, and by Assumption \ref{Assumption 1}, the sub-level set $lev_{f_j(\bm x_0)}(f_j)$ is bounded. Consequently,  the sequence $\{\bm x_\nu\}_\nu$ has at least one accumulation point, say $\bm x^*$, i.e., there exists $\mathcal{V}\subset\mathbb{N}_0$ such that $\bm x_\nu\xrightarrow{\mathcal{V}}\bm x^*$. This fact along with \eqref{T3-1} implies 
\begin{equation}\label{T3-2}
\min\big\{\lVert \boldsymbol{\xi}\lVert \,\, : \,\, \boldsymbol{\xi}\in\text{\rm \texttt{conv}}\{ \cup_{i=1}^p\bm{\mathcal{S}}^{k_\nu}_{\varepsilon_\nu,i}(\bm x_{\nu+1})\}   \big\}\leq \delta_\nu, \quad \text{for all}\,\, \nu\in\mathcal{V}. 
\end{equation}
Let $\kappa>0$ be arbitrary. As $\delta_\nu\downarrow 0$ as $\nu\to\infty$, one can find $\bar{\nu}\in\mathcal{V}$ sufficiently large such that $\delta_\nu<\kappa$, for all $\nu\geq \bar{\nu}$. Let $\mathcal{V}':=\{\nu\in\mathcal{V} \, : \, \nu\geq \bar{\nu}  \}\subset\mathcal{V}$. Then, it immediately follows from \eqref{T3-2} that, for all  $\nu\in\mathcal{V}'$,
\begin{equation}\label{T3-3}
\lVert \boldsymbol{\xi}^*_\nu\rVert:=
	\min\big\{\lVert \boldsymbol{\xi}\lVert \,\, : \,\, \boldsymbol{\xi}\in\text{\rm \texttt{conv}}\{ \cup_{i=1}^p\bm{\mathcal{S}}^{k_\nu}_{\varepsilon_\nu,i}(\bm x_{\nu+1})\}   \big\}\leq \kappa.
\end{equation}
In addition, since 
$$\boldsymbol{\xi}^*_\nu\in\text{\rm \texttt{conv}}\{ \cup_{i=1}^p\bm{\mathcal{S}}^{k_\nu}_{\varepsilon_\nu,i}(\bm x_{\nu+1})\}\subset \text{\rm \texttt{conv}}\{ \cup_{i=1}^p\partial_{\varepsilon_{\nu}} f_i(\bm x_{\nu+1})\}, \quad \text{for all} \,\,  \nu\in\mathcal{V}',  $$
there exist
 $$\boldsymbol{\xi}^*_{\nu,i}\in\partial_{\varepsilon_\nu}f_i(\bm x_{\nu+1}), \,\, \lambda_{\nu,i}\in\mathbb{R}, \quad i=1,\ldots,p,$$
 satisfying 
 \begin{equation}\label{T3-4}
 \boldsymbol{\xi}^*_\nu=\sum_{i=1}^{p}\lambda_{\nu,i} \boldsymbol{\xi}^*_{\nu,i}, \quad \sum_{i=1}^{p} \lambda_{\nu,i}=1, \quad \lambda_{\nu,i}\geq 0, \,\,\, i=1,\ldots,p,
 \end{equation}
 for all  $\nu\in\mathcal{V}'$. The fact that $\bm x_\nu\xrightarrow{\mathcal{V}'}\bm x^*$ along with uniformly compactness of the map $\partial_\cdot f_i(\cdot)$ implies that the sequences $\{\boldsymbol{\xi}^*_{\nu,i} \}_{\nu\in\mathcal{V}'}$, for $i=1,\ldots,p$, are bounded. Now, boundedness of the sequences $\{\lambda_{\nu,i} \}_{\nu\in\mathcal{V}'}$, for $i=1,\ldots,p$, implies the existence of the infinite set  $\mathcal{V}''\subset\mathcal{V}'$ such that 
 \begin{equation}\label{T3-5}
 \boldsymbol{\xi}^*_{\nu,i}\xrightarrow{\mathcal{V}''}\boldsymbol{\xi}^*_i, \quad \lambda_{\nu,i}\xrightarrow{\mathcal{V}''}\lambda_i, \quad \lambda_i\geq 0, \quad \sum_{i=1}^{p}\lambda_{i}=1,   \quad \text{for all} \,\, i=1,\ldots,p.
 \end{equation}
 Furthermore, as the map $\partial_\cdot f_i(\cdot)$ is  upper semicontinuous, one can deduce 
 \begin{equation}\label{T3-6}
 \boldsymbol{\xi}^*_i\in\partial f_i(\bm x^*), \quad \text{for all} \,\, i=1,\ldots,p.
 \end{equation}
 Eventually, in view of \eqref{T3-4}, \eqref{T3-5} and \eqref{T3-6}, we conclude 
 \begin{equation}\label{T3-7}
 \boldsymbol{\xi}^*_\nu=\sum_{i=1}^{p} \lambda_{\nu,i} \boldsymbol{\xi}^*_{\nu,i}  \xrightarrow{\mathcal{V}''} \sum_{i=1}^{p} \lambda_i \boldsymbol{\xi}^*_i\in\texttt{conv}\{ \cup_{i=1}^p \partial f_i(\bm x^*)  \}.
 \end{equation}
 Consequently, \eqref{T3-3} and \eqref{T3-7} yield
 $$ \min\big\{\lVert \boldsymbol{\xi}\rVert \,\, : \,\, \boldsymbol{\xi}\in \texttt{conv}\{ \cup_{i=1}^p \partial f_i(\bm x^*)  \}\big\}\leq \kappa. $$
Since $\kappa>0$ was arbitrary, we conclude $\bm 0\in \texttt{conv}\{ \cup_{i=1}^p \partial f_i(\bm x^*)  \}$.
\end{proof}

\begin{remark}
	Concerning Remark 4.1 in \cite{Multi-sub4}, our proof of Theorem \ref{T3} shows the convergence of Algorithm 4 of Gebken and Peitz \cite{Multi-sub4} to a Clarke substationary point when $\delta_{\nu}\downarrow 0 , \varepsilon_{\nu} \downarrow 0$ as $\nu\to\infty$.
\end{remark}

\section{Numerical experiments} \label{Sec6}
In this section, we assess the practical performance of the proposed multi-objective optimization method. To this end, first we analyze the typical behavior of the proposed method. Next, we apply our method to approximate the Pareto fronts of five multiobjective test problems. Then, we seek a sparse solution to an underdetermined linear system by considering a nonsmooth bi-objective minimization problem. Eventually, we apply the proposed method to a collection of nonsmooth convex and nonconvex test problems, and some comparative results with a state-of-the-art algorithm are provided. Throughout this section, the proposed method is called \textbf{SUMOPT} (\textbf{SU}bgradient method for \textbf{M}ultiobjective \textbf{OPT}imization). We have implemented the method in \textsc{Matlab} (R2017b) on a machine with CPU 3.5 GHz and 8 GB of RAM.

 Regarding the parameters for \textbf{SUMOPT}, our choices are as follows.
 
 \begin{itemize}
 	\item Algorithm \ref{Alg1}. In the \texttt{LBLS} algorithm, the reduction factor $r\in(0,1)$ and the initial step size $t_0\in(\bar{t},\infty)$ are set to be $0.5$ and 2.0, respectively. In addition, for the number of backtrack steps $\tau\in\mathbb{N}$, we choose
 	$\tau:=\lceil{\frac{\ln{\bar{t}}-\ln{t_0}}{\ln{r}}-1}\rceil$. 
 	\item Algorithm \ref{Line-Search}. In the \texttt{FES} algorithm, we set $c:=0.01$. Moreover, since $\eta\in(0,0.5)$, in Line \ref{Line11} of this algorithm, one can update $t_s$ by $$t_{s+1}:=0.5 (t^u+t^l).$$
 	\item Algorithm \ref{Alg3}. In this algorithm, the lower bound $\bar t\in(0, \varepsilon)$ is set to $\varepsilon/10$. Also, the quadratic subproblem in this algorithm is solved using the \emph{quadprog} solver in \textsc{Matlab}.
 	\item Algorithm \ref{Alg4}. We work with the Armijo coefficient $\beta:=10^{-6}$, and the positive sequences $\delta_{\nu+1}:=\gamma \delta_{\nu}$, $\varepsilon_{\nu+1}:=\gamma \varepsilon_{\nu}$ with $\gamma=\delta_0=\varepsilon_0:=0.1$. Moreover, the starting point $\bm x_0\in\mathbb{R}^n$ and optimality tolerance $\rho>0$ will be specified in each experiment.

 \end{itemize}
 The parameters mentioned earlier were utilized in all the subsequent experiments, except for subsection \ref{Subsect 5.1}. In this particular case, we ran the \textbf{SUMOPT} method using $\gamma=0.5$, $\delta_0=0.3, \bar{t}:={\varepsilon}/2 $, and $t_0=0.25$ to increase the number of iterations and obtain a more comprehensive understanding of the method's performance.
 
 In the following experiments, we consider a collection of convex and nonconvex multiobjective test problems with $n=2$ and $p\in\{2,3,4,5\}$, which are presented in Table \ref{Table0}. In this table, ``P'' stands for the problem number, and the symbols ``\texttt{+}'' and ``\texttt{-}'' denote convex and nonconvex objectives, respectively. For more details on the considered objective functions, see \cite{Bagirov2014,Makela_book}. The remaining test problems are provided and described in the corresponding subsections.

 \begin{table}[h]
 	\caption{List of test problems}\label{Table0}
 	\resizebox{\textwidth}{!}{%
 		\begin{tabular}{lcllllll}
 			\toprule[1.5pt]
 			& \rule{0pt}{3ex}P &  & \multicolumn{1}{c}{$f_i$} &  & \multicolumn{1}{c}{Convexity} &   \\ 
 			\cline{2-2}\cline{4-4}\cline{6-6}\cline{8-8}
 			&\rule{0pt}{3ex}\texttt{1} &  & \texttt{[Crescent, LQ]} &  & \texttt{[-, +]} &    \\ 
 			& \texttt{2} &  &\texttt{[Mifflin2, Crescent]}  &  & \texttt{[-, -]} &    \\ 
 			& \texttt{3} &  &\texttt{[Crescent, QL]}  &  & \texttt{[-, +]} &    \\ 
 			& \texttt{4} &  & \texttt{[CB3, LQ]} &  & \texttt{[+, +]} &    \\ 
 			& \texttt{5} &  & \texttt{[CB3, Mifflin1]} &  & \texttt{[+, +]} &    \\ 
 			& \texttt{6} &  & \texttt{[Mifflin2, Mifflin1]} &  & \texttt{[-, +]} &    \\ 
 			& \texttt{7} &  & \texttt{[CB3, QL]} &  & \texttt{[+, +]} &    \\ 
 			& \texttt{8} &  & \texttt{[Mifflin2, DEM]} &  & \texttt{[-, +]} &    \\ 
 			& \texttt{9} &  & \texttt{[Mifflin2, LQ]} &  & \texttt{[-, +]} &    \\ 
 			& \texttt{10} &  & \texttt{[CB3, DEM]} &  & \texttt{[+, +]} &    \\ 
 			& \texttt{11} &  & \texttt{[DEM, QL, Mifflin1]} &  & \texttt{[+, +, +]} &    \\ 
 			& \texttt{12} &  & \texttt{[Mifflin2, Crescent, Mifflin1]} &  & \texttt{[-, -, +]} &    \\ 
 			& \texttt{13} &  & \texttt{[DEM, QL, Mifflin1, CB3]} &  & \texttt{[+, +, +, +]} &   \\ 
 			& \texttt{14} &  & \texttt{[Mifflin2, Crescent, DEM, Mifflin1]} &  & \texttt{[-, -, +, +]} &    \\ 
 			& \texttt{15} &  & \texttt{[Mifflin2, Crescent, DEM, Mifflin1, QL]} &  & \texttt{[-, -, +, +, +]} &    \\ 
 			
 			\toprule[1.5pt]
 		\end{tabular} 
 	}
 \end{table}
 Throughout this section, we also consider a state-of-the-art algorithm, namely the \textbf{M}ultiobjective \textbf{P}roximal \textbf{B}undle method (\textbf{MPB}) \cite{techreport-multi, MPB-Techreport} to provide some comparative results. For this purpose, we used the \textsc{Fortran} code of the \textbf{MPB} method  from \cite{MPB-Techreport} using the default parameters.

 \subsection{Tracking the algorithm} \label{Subsect 5.1}
 In this experiment, we analyze the typical behavior of the proposed  \textbf{SUMOPT} algorithm using the test problem P1. Indeed, we consider the following bi-objective test problem:
 \begin{equation}\label{P1}
 \min \,\, \left(f_1(\bm x), f_2(\bm x)\right) \quad \text{s.t.} \quad \bm x\in\mathbb{R}^2,
 \end{equation} 
 in which 
 \begin{align*}
 &f_1(\bm x):=\max\{x_1^2+(x_2-1)^2+x_2-1, -x_1^2-(x_2-1)^2+x_2+1\} \quad \text{(Crescent),} \\&
 f_2(\bm x):=\max\{-x_1-x_2, -x_1-x_2+x_1^2+x_2^2-1 \} \quad  \text{(LQ).}
 \end{align*}
 It is evident that $f_1(\bm x)$ is a nonconvex function, while $f_2(\bm x)$ is convex, and the nonsmooth regions of these functions are given by 
 $$\mathcal{M}_1:=\{\bm x\in\mathbb{R}^2 \,:\, x_1^2+(x_2-1)^2=1  \} \quad \text{and}\quad \mathcal{M}_2:=\{\bm x\in\mathbb{R}^2 \,:\, x_1^2+x_2^2=1  \},$$
 respectively.
  Table~\ref{Table1} reports the observed results using the starting point $\bm x_0^T=[-0.6, 0.2]$ and the optimality tolerance $\rho=5\times10^{-3}$. Each block of the table displays the results of Algorithm \ref{Alg3}, which are generated by Algorithm \ref{Alg4} at iteration $\nu\geq0$.
  
  \begin{table}
  	\caption{Numerical results of \textbf{SUMOPT} applied to problem \eqref{P1}.}\label{Table1}
  	\resizebox{\textwidth}{!}{%
  		\begin{tabular}{cllllllllllll}
  			\toprule[1.5pt]
  			\multicolumn{13}{c}{\rule{0pt}{3ex}$ \nu=0 \qquad \delta_{0}=0.3 \qquad \varepsilon_{0}=0.1$} \\ 
  			\toprule[1.5pt]
  			\rule{0pt}{3ex}$k$ &  & \multicolumn{1}{c}{$\lVert \boldsymbol{\xi}^*_k\rVert$} &  & \multicolumn{1}{c}{$\bm d_k^T$} &  & \multicolumn{1}{c}{$\tilde{I}_k$} &  & \multicolumn{1}{c}{$\bm x_{k+1}^T$} &  & \multicolumn{1}{c}{$f_1(\bm x_{k+1})$} &  & \multicolumn{1}{c}{$f_2(\bm x_{k+1})$} \\ 
  			\cline{1-1}\cline{3-3}\cline{5-5}\cline{7-7}\cline{9-9}\cline{11-11}\cline{13-13}
  			\rule{0pt}{3ex}\texttt{0} &  & \multicolumn{1}{c}{\texttt{1.3416}} &  & \texttt{(0.8944, 0.4472)} &  & \multicolumn{1}{c}{$\{\texttt{1} \}$} &  & \texttt{(-0.6000, 0.2000)} &  & \multicolumn{1}{c}{\texttt{0.2000}} &  & \multicolumn{1}{c}{\texttt{0.4000}} \\ 
  			\texttt{1} &  & \multicolumn{1}{c}{\texttt{0.3494}} &  & \texttt{(0.8598, -0.5104)} &  & \multicolumn{1}{c}{$\varnothing$} &  & \texttt{(-0.3850, 0.0723)} &  & \multicolumn{1}{c}{\texttt{0.0811}} &  & \multicolumn{1}{c}{\texttt{0.3126}} \\ 
  			\texttt{2} &  & \multicolumn{1}{c}{\texttt{1.1508}} &  & \texttt{(0.6691, 0.7431)} &  & \multicolumn{1}{c}{$\{\texttt{1} \}$} &  & \texttt{(-0.3850, 0.0723)} &  & \multicolumn{1}{c}{\texttt{0.0811}} &  & \multicolumn{1}{c}{\texttt{0.3126}} \\ 
  			\texttt{3} &  & \multicolumn{1}{c}{\texttt{0.3925}} &  & \multicolumn{1}{c}{\texttt{(0.9268, -0.3755)}} &  & \multicolumn{1}{c}{$\varnothing$} &  & \texttt{(-0.1533, -0.0214)} &  & \multicolumn{1}{c}{\texttt{0.0454}} &  & \multicolumn{1}{c}{\texttt{0.1748}} \\ 
  			\texttt{4} &  & \multicolumn{1}{c}{\texttt{1.0871}} &  & \texttt{(0.2820, 0.9593)} &  & \multicolumn{1}{c}{$\varnothing$} &  & \texttt{(-0.1445, 0.0084)} &  & \multicolumn{1}{c}{\texttt{0.0124}} &  & \multicolumn{1}{c}{\texttt{0.1360}} \\ 
  			\texttt{5} &  & \multicolumn{1}{c}{\texttt{1.0246}} &  & \texttt{(0.2820, 0.9593)} &  & \multicolumn{1}{c}{$\{\texttt{1} \}$} &  & \texttt{(-0.1445, 0.0084)} &  & \multicolumn{1}{c}{\texttt{0.0124}} &  & \multicolumn{1}{c}{\texttt{0.1360}} \\ 
  			\texttt{6} &  & \multicolumn{1}{c}{\texttt{0.1478}} &  & \multicolumn{1}{c}{\textbf{--}} &  & \multicolumn{1}{c}{\textbf{--}} &  & \multicolumn{1}{c}{\textbf{--}} &  & \multicolumn{1}{c}{\textbf{--}} &  & \multicolumn{1}{c}{\textbf{--}} \\ 
  			\toprule[1.5pt]
  			\multicolumn{13}{c}{\rule{0pt}{3ex}$\nu=1 \qquad \delta_{1}=0.15 \qquad \varepsilon_{1}=0.05$} \\ 
  			\toprule[1.5pt]
  			\rule{0pt}{3ex}\texttt{0} &  & \multicolumn{1}{c}{\texttt{1.0246}} &  & \texttt{(0.2820, 0.9593)} &  & \multicolumn{1}{c}{$\{\texttt{1} \}$} &  & \texttt{(-0.1445, 0.0084)} &  & \multicolumn{1}{c}{\texttt{0.0124}} &  & \multicolumn{1}{c}{\texttt{0.01360}} \\ 
  			\texttt{1} &  & \multicolumn{1}{c}{\texttt{0.1460}} &  & \multicolumn{1}{c}{\textbf{--}} &  & \multicolumn{1}{c}{\textbf{--}} &  & \multicolumn{1}{c}{\textbf{--}} &  & \multicolumn{1}{c}{\textbf{--}} &  & \multicolumn{1}{c}{\textbf{--}} \\ 
  			\toprule[1.5pt]
  			\multicolumn{13}{c}{\rule{0pt}{3ex}$\nu=2 \qquad \delta_{2}=0.075 \qquad \varepsilon_{2}=0.025$} \\ 
  			\toprule[1.5pt]
  			\rule{0pt}{3ex}\texttt{0} &  & \multicolumn{1}{c}{\texttt{1.0246}} &  & \texttt{(0.2820, 0.9593)} &  & \multicolumn{1}{c}{$\varnothing$} &  & \texttt{(-0.1439, 0.0103)} &  & \multicolumn{1}{c}{\texttt{0.0104}} &  & \multicolumn{1}{c}{\texttt{0.1335}} \\ 
  			\texttt{1} &  & \multicolumn{1}{c}{\texttt{1.0207}} &  & \texttt{(0.2820, 0.9593)} &  & $\{\texttt{1} \}$ &  & \texttt{(-0.1439, 0.0103)} &  & \multicolumn{1}{c}{\texttt{0.0104}} &  & \multicolumn{1}{c}{\texttt{0.1335}} \\ 
  			\texttt{2} &  & \multicolumn{1}{c}{\texttt{0.1448}} &  & \texttt{(0.9897, -0.1430)} &  & \multicolumn{1}{c}{$\varnothing$} &  & \texttt{(-0.0224, -0.0075)} &  & \multicolumn{1}{c}{\texttt{0.0079}} &  & \multicolumn{1}{c}{\texttt{0.0277}} \\ 
  			\texttt{3} &  & \multicolumn{1}{c}{\texttt{1.0158}} &  & \texttt{(0.0398, 0.9992)} &  & \multicolumn{1}{c}{$\varnothing$} &  & \texttt{(-0.0199, 0.0002)} &  & \multicolumn{1}{c}{\texttt{0.0004}} &  & \multicolumn{1}{c}{\texttt{0.0196}} \\ 
  			\texttt{4} &  & \multicolumn{1}{c}{\texttt{0.7161}} &  & \texttt{(0.9678, -0.2516)} &  & \multicolumn{1}{c}{$\{\texttt{1} \}$} &  & \texttt{(-0.0199, 0.0002)} &  & \multicolumn{1}{c}{\texttt{0.0004}} &  & \multicolumn{1}{c}{\texttt{0.0196}} \\ 
  			\texttt{5} &  & \multicolumn{1}{c}{\texttt{0.0017}} &  & \multicolumn{1}{c}{\textbf{--}} &  & \multicolumn{1}{c}{\textbf{--}} &  & \multicolumn{1}{c}{\textbf{--}} &  & \multicolumn{1}{c}{\textbf{--}} &  & \multicolumn{1}{c}{\textbf{--}} \\ 
  			\toprule[1.5pt]
  			\multicolumn{13}{c}{\rule{0pt}{3ex}$\nu=3 \qquad \delta_{3}=0.0375 \qquad \varepsilon_{3}=0.0125$} \\ 
  			\toprule[1.5pt]
  			\rule{0pt}{3ex}\texttt{0} &  & \multicolumn{1}{c}{\texttt{0.7161}} &  & \texttt{(0.9678, -0.2516)} &  & \multicolumn{1}{c}{$\varnothing$} &  & \texttt{(-0.0189, 0.0000)} &  & \multicolumn{1}{c}{\texttt{0.0003}} &  & \multicolumn{1}{c}{\texttt{0.0189}} \\ 
  			\texttt{1} &  & \multicolumn{1}{c}{\texttt{1.0006}} &  & \texttt{(0.0379, 0.9992)} &  & \multicolumn{1}{c}{$\{\texttt{1} \}$} &  & \texttt{(-0.0189, 0.0000)} &  & \multicolumn{1}{c}{\texttt{0.0003}} &  & \multicolumn{1}{c}{\texttt{0.0189}} \\ 
  			\texttt{2} &  & \multicolumn{1}{c}{\texttt{0.0190}} &  & \multicolumn{1}{c}{\textbf{--}} &  & \multicolumn{1}{c}{\textbf{--}} &  & \multicolumn{1}{c}{\textbf{--}} &  & \multicolumn{1}{c}{\textbf{--}} &  & \multicolumn{1}{c}{\textbf{--}} \\ 
  			\toprule[1.5pt]
  			\multicolumn{13}{c}{\rule{0pt}{3ex}$\nu=4 \qquad \delta_{4}=0.01875 \qquad \varepsilon_{4}=0.00625$} \\ 
  			\toprule[1.5pt]
  			\rule{0pt}{3ex}\texttt{0} &  & \multicolumn{1}{c}{\texttt{1.0006}} &  & \texttt{(0.0379, 0.9992)} &  & \multicolumn{1}{c}{$\{\texttt{1} \}$} &  & \texttt{(-0.0189, -0.0000)} &  & \multicolumn{1}{c}{\texttt{0.0003}} &  & \multicolumn{1}{c}{\texttt{0.0189}} \\ 
  			\texttt{1} &  & \multicolumn{1}{c}{\texttt{0.0190}} &  & \texttt{(0.9998, -0.0189)} &  & \multicolumn{1}{c}{$\varnothing$} &  & \texttt{(-0.0033, -0.0002)} &  & \multicolumn{1}{c}{\texttt{0.0002}} &  & \multicolumn{1}{c}{\texttt{0.0036}} \\ 
  			\texttt{2} &  & \multicolumn{1}{c}{\texttt{1.0005}} &  & \texttt{(0.0067, 0.9999)} &  & \multicolumn{1}{c}{$\{\texttt{1} \}$} &  & \texttt{(-0.0033, -0.0002)} &  & \multicolumn{1}{c}{\texttt{0.0002}} &  & \multicolumn{1}{c}{\texttt{0.0036}} \\ 
  			\texttt{3} &  & \multicolumn{1}{c}{\texttt{0.0033}} &  & \multicolumn{1}{c}{\textbf{--}} &  & \multicolumn{1}{c}{\textbf{--}} &  & \multicolumn{1}{c}{\textbf{--}} &  & \multicolumn{1}{c}{\textbf{--}} &  & \multicolumn{1}{c}{\textbf{--}} \\ 
  			\toprule[1.5pt]
  			\multicolumn{13}{c}{\rule{0pt}{3ex}$\nu=5 \qquad \delta_{5}=0.009375\qquad \varepsilon_{5}=0.003125$} \\ 
  			\toprule[1.5pt]
  			\rule{0pt}{3ex}\texttt{0} &  & \multicolumn{1}{c}{\texttt{1.0005}} &  & \texttt{(0.0067, 0.9999)} &  & \multicolumn{1}{c}{$\varnothing$} &  & \texttt{(-0.0033, -0.0000)} &  & \multicolumn{1}{c}{\texttt{0.0000}} &  & \multicolumn{1}{c}{\texttt{0.0033}} \\ 
  			\texttt{1} &  & \multicolumn{1}{c}{\texttt{1.0000}} &  & \texttt{(0.0067, 0.9999)} &  & \multicolumn{1}{c}{$\{\texttt{1} \}$} &  & \texttt{(-0.0033, -0.0000)} &  & \multicolumn{1}{c}{\texttt{0.0000}} &  & \multicolumn{1}{c}{\texttt{0.0033}} \\ 
  			\texttt{2} &  & \multicolumn{1}{c}{\texttt{0.0033}} &  & \multicolumn{1}{c}{\textbf{--}} &  & \multicolumn{1}{c}{\textbf{--}} &  & \multicolumn{1}{c}{\textbf{--}} &  & \multicolumn{1}{c}{\textbf{--}} &  & \multicolumn{1}{c}{\textbf{--}} \\ 
  			\toprule[1.5pt]
  			\multicolumn{13}{c}{\rule{0pt}{3ex}$\nu=6 \qquad \delta_{6}=0.0046875\qquad \varepsilon_{6}=0.0015625$} \\ 
  			\toprule[1.5pt]
  			\multicolumn{13}{c}{\rule{0pt}{3ex} \texttt{Termination} } \\
  			\toprule[1.5pt]
  		\end{tabular}
  	}
  \end{table}

   For $\nu=0$, since $\bm x_0\in\mathcal{M}_1$, we observe $\tilde{I}_0=\{1 \}$, which means the singleton set $\mathcal{S}_{0.1,1}^0$ is not an adequate approximation of $\partial_{0.1} f_1(\bm x_0)$ to make a reduction in objective function $f_1$. Consequently, the method admits a null step to enrich $\mathcal{S}_{0.1,1}^0$  by using the \texttt{FES} algorithm. Note that, as $2\notin\tilde{I}_0$, the \texttt{LBLS} algorithm has already obtained a descent step for $f_2$ at $\bm x_0$, and hence we do not improve our approximation of  $\partial_{0.1} f_2(\bm x_0)$ at iteration $k=0$. Once a new effective element of $\partial_{0.1} f_1(\bm x_0)$  is appended to $\mathcal{S}_{0.1,1}^0$ to form $\mathcal{S}_{0.1,1}^1$, one can see $\tilde{I}_1$ becomes empty, which means a serious step occurs at $k=1$. Moreover, since $\mathcal{S}_{0.1,1}^0\subsetneq\mathcal{S}_{0.1,1}^1$, we have $\lVert \boldsymbol{\xi}^*_1\rVert<\lVert \boldsymbol{\xi}^*_0\rVert$. In general, after each serious step, $\lVert \boldsymbol{\xi}^*_k\rVert$ becomes smaller. However, it should be noted  that the sequence $\{\lVert \boldsymbol{\xi}^*_k\rVert \}_k$ is not monotonically decreasing, yet, as shown in the proof of Theorem \ref{T3}, it has a decreasing subsequence. Eventually, at iteration $k=6$, $\lVert \boldsymbol{\xi}^*_6\rVert$ drops below $\delta_0=0.3$, and thus Algorithm \ref{Alg3} terminates.
   
  Regarding $\nu=1$, we do not see any serious step. In fact, after a null step, the method recognizes that the current point $\bm x_1$ is indeed a $\big(0.15, \mathcal{S}_{0.05}(\bm x_1)\big)$-substationary point (see Definition \ref{Def1}).
  
   For $\nu=2$, Algorithm~\ref{Alg3} is executed with starting point $\bm x_0^T=(-0.1445, 0.0084)$, at which the search direction $\bm d_0^T=(0.2820, 0.9593)$ is obtained.
Although no serious step is taken along this direction at $\nu=1$, for $\nu=2$ we observe $\tilde{I}_0=\varnothing$, indicating that a serious step is taken. This is because as $\nu$ increases, the lower bound $\bar{t}=\varepsilon_{\nu}/{2}$ in the \texttt{LBLS} algorithm decreases, resulting in more backtrack steps. Consequently, a serious step occurs. After this serious step, the current point $\bm x_1$ approaches the nonsmooth region $\mathcal{M}_1$. Hence, taking a null step allows the \texttt{FES} algorithm to capture a subgradient from the opposite side of $\mathcal{M}_1$, which significantly modifies the search direction $\bm d^T_1$. In fact, the modified search direction $\bm d^T_2$ is almost perpendicular to $\bm d^T_1$, resulting in an effective serious step. The method continues in this manner until, at $\nu=6$, we observe $\delta_6$ and $\varepsilon_6$ become smaller than the optimality tolerance $\rho=0.5\times10^{-3}$, and the \textbf{SUMOPT} algorithm terminates.
 
 \subsection{Approximating the Pareto front}\label{Subsec}
 In addition to convergence, the efficiency of a method in approximating the Pareto front is an important property. Thus, the main goal of this experiment is to employ the multi-start technique \cite{Multi-start1,Multi-start2} in order to generate an approximation of the Pareto front. To this end, we consider problems P1-P5 given in Table \ref{Table0}. To generate different Pareto points using the \textbf{SUMOPT} and \textbf{MPB} algorithms, we independently and uniformly generate $3\times10^{2}$ random starting points from the following two-dimensional box: 
 
 $$\bm{B}:=\{(x_1, x_2)\in\mathbb{R}^2 \,\, : \,\, 0 \leq x_1, x_2\leq2 \}.$$
 Moreover, the optimality tolerance $\rho$ in the \textbf{SUMOPT} method is set to $10^{-4}$.
  Figure \ref{Fig1} depicts the objective space of problems P1-P5 along with the obtained approximations of the Pareto fronts by the \textbf{SUMOPT} (top plots)  and \textbf{MPB} (bottom plots) methods.


 \begin{figure}[h]
 	\centering
 	\includegraphics[width=\textwidth]{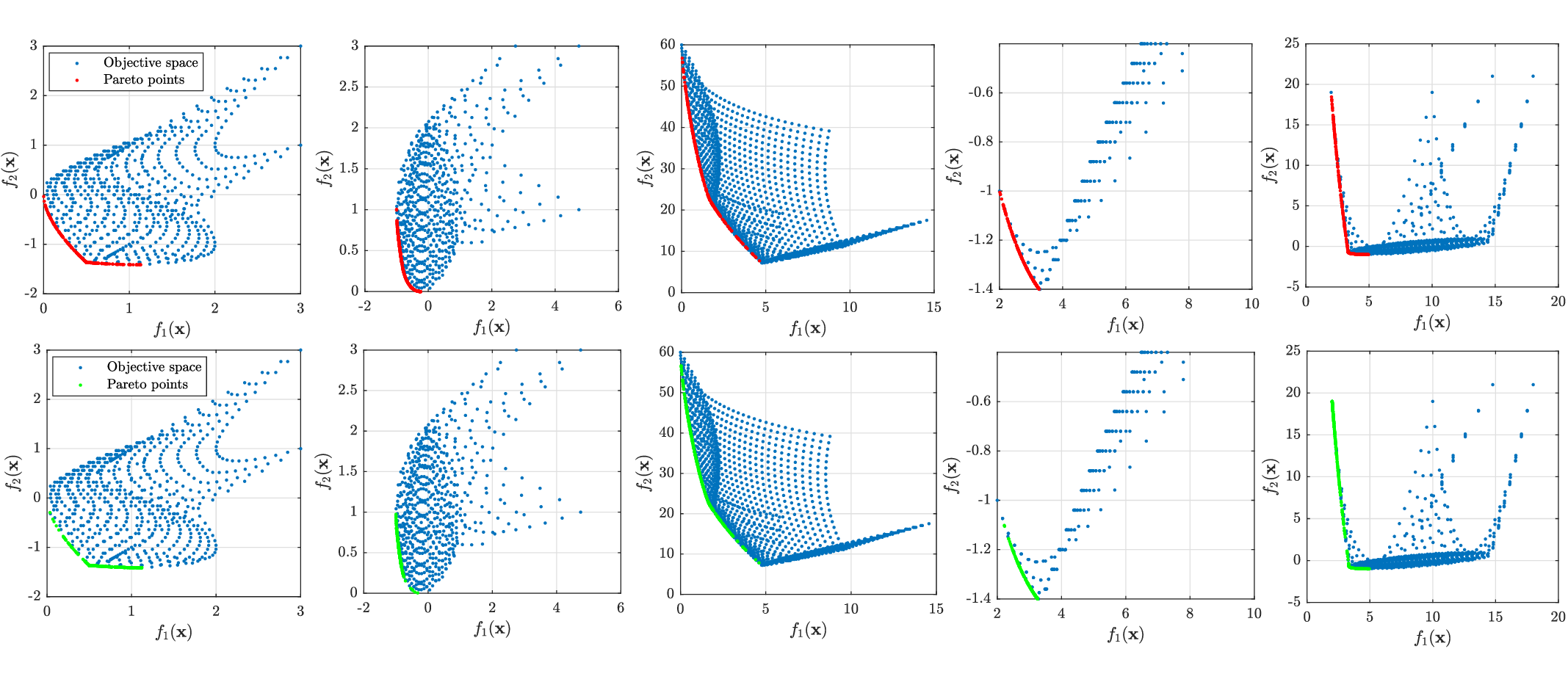}
 	\caption{Objective space  and obtained approximations of the Pareto fronts by the \textbf{SUMOPT} (top plots) and \textbf{MPB} (bottom plots) methods, for problems P1-P5.}\label{Fig1}
 \end{figure}
 
 Once the Pareto set of a bi-objective optimization problem is connected, the \textbf{H}ole \textbf{A}bsolute \textbf{S}ize (\textbf{HAS}) and \textbf{H}ole \textbf{R}elative \textbf{S}ize (\textbf{HRS}) metrics \cite{Multi-Indicators} are effective tools for assessing the uniformity  of the points along the approximate Pareto front. Assume that $\mathcal{Y}$ is a Pareto set approximation whose elements have been sorted in ascending order based on the first objective function. For $j=1,\ldots,\lvert \mathcal{Y}\rvert-1$, take $d_j:=\lVert \bm y_j-\bm y_{j+1}\rVert$, which is the Euclidean distance between two consecutive solutions $\bm y_j, \bm y_{j+1}\in\mathcal{Y}$. Then, if $\bar{D}$ is the mean of all $d_j$ for $j=1,\ldots,\lvert \mathcal{Y}\rvert-1$,  the indicators \textbf{HAS} and \textbf{HRS} are defined by 
 \begin{align*}
 &\textbf{HAS}(\mathcal{Y}):=\max \{d_j \,\,:\,\, j=1,\ldots,\lvert \mathcal{Y}\rvert-1 \}, \\&
 \textbf{HRS}(\mathcal{Y}):=\frac{1}{\bar{D}}\max \{d_j \,\,:\,\, j=1,\ldots,\lvert \mathcal{Y}\rvert-1 \}.
 \end{align*}
 The \textbf{HAS} measures the absolute size of the largest gap between consecutive solutions, while the \textbf{HRS} provides a relative measure of this gap compared to the average gap size. Obviously, for these indicators, a lower value is desirable.

 Table \ref{Table2} presents the metrics values and the cardinality of the approximate Pareto fronts, $\lvert\mathcal{Y}\rvert$, for the obtained approximations addressed in Figure \ref{Fig1}. By comparing the values of \textbf{HAS} and \textbf{HRS} between the two methods, one can conclude that \textbf{SUMOPT} offers a more uniform approximation of the Pareto fronts in the considered set of test problems. In addition, it can be seen that for problem P2 and the \textbf{MPB} method, sixty-two of the starting points did not result in Pareto points. The same happened for problem P3 and the \textbf{SUMOPT} method for seven of the starting points.
 
 \begin{table}
 	\caption{Metrics  \textbf{HAS} and \textbf{HRS} along with $\lvert\mathcal{Y}\rvert$, for problems P1-P5.} \label{Table2}
 	\resizebox{\textwidth}{!}{%
 		
 		\begin{tabular}{cllllllcllllllcl}
 			\toprule[1.5pt]
 			& \multicolumn{1}{c}{} &  & \multicolumn{5}{c}{$\textbf{SUMOPT}$} &  & \multicolumn{1}{c}{} & \multicolumn{5}{c}{$\textbf{MPB}$} &  \\ 
 			\cline{4-8}\cline{11-15}
 			\rule{0pt}{3ex}P &  &  & \multicolumn{1}{c}{\textbf{HAS}} &  & \multicolumn{1}{c}{\textbf{HRS}} &  & $\lvert\mathcal{Y}\rvert$ &  &  & \multicolumn{1}{c}{\textbf{HAS}} &  & \multicolumn{1}{c}{\textbf{HRS}} &  & $\lvert\mathcal{Y}\rvert$ &  \\ 
 			\cline{1-1}\cline{4-4}\cline{6-6}\cline{8-8}\cline{11-11}\cline{13-13}\cline{15-15}
 			\multicolumn{1}{l}{\rule{0pt}{3ex}$\texttt{1}$} &  &  & \texttt{0.0952} &  & \texttt{13.6061} &  & \texttt{300} &  &  & \texttt{0.1403} &  & \texttt{23.3442} &  & \texttt{300} &  \\ 
 			$\texttt{2}$ &  &  & \texttt{0.1379} &  & \texttt{29.0251} &  & \texttt{300} &  &  & \texttt{0.1846} &  & \texttt{31.7277} &  & \texttt{238} &  \\ 
 			$\texttt{3}$ &  &  & \texttt{1.5664} &  & \texttt{9.2399} &  & \texttt{293} &  &  & \texttt{2.4199} &  & \texttt{14.7110} &  & \texttt{300} &  \\ 
 			$\texttt{4}$ &  &  & \texttt{0.0544} &  & \texttt{11.6060} &  & \texttt{300} &  &  & \texttt{0.1334} &  & \texttt{34.2662} &  & \texttt{300} &  \\ 
 			$\texttt{5}$ &  &  & \texttt{0.6107} &  & \texttt{8.7263} &  & \texttt{300} &  &  & \texttt{1.6905} &  & \texttt{23.5538} &  & \texttt{300} &  \\ 
 			\toprule[1.5pt]
 		\end{tabular}
 	}
 \end{table}
 
 Our next example presents an interesting case. It involves a smooth bi-objective problem where the objective space is not $\mathbb{R}^p_{\geqq}$-convex \cite{Ehrgott2005}, and the set of Pareto front is not connected. In such a scenario, it is well-known that the \textbf{W}eighted \textbf{S}um (\textbf{WS}) method \cite{Ehrgott2005,Geof}, as a popular scalarization technique, falls short of generating a suitable approximation of the Pareto front. Let us design the following bi-objective test problem:
 \begin{equation}\label{FL}
 \min \,\, \left(f_1(x), f_2(x)\right) \quad \text{s.t.} \quad  x\in\mathbb{R},
 \end{equation} 
 where $f_1,f_2:\mathbb{R}\to\mathbb{R}$ are given by
 
 \begin{align*}
 &f_1(x):= \big(1 + 0.1 \sin(8x)\big)\big(\cos(x)\cos(0.6) - \sin(-0.6) \sin(x) \big),\\&
 f_2(x):= \big(1 + 0.1 \sin(8x)\big)\big(\cos(x)\sin(-0.6) + \sin(x) \cos(0.6) \big).
 \end{align*}
  Since $f_1$ and $f_2$ are periodic functions with periodicity $T=2\pi$, the objective space of problem  \eqref{FL} can be obtained, e.g., on the interval $I=[0, 2\pi]$, see Figure~\ref{Fig1'} (top-left). Let $G(x):=\nabla f_1(x) \nabla f_2(x)$. Then, it is easy to verify that the set of substationary points of problem \eqref{FL} on the interval $I$ is given by ${S}:= \{x\in I \,\,: \,\, G(x)\leq 0  \}$. The top-middle plot of Figure~\ref{Fig1'} depicts the graph of $G(x)$ together with its nonpositive parts, $G^{-}(x)$, associated with the set $S$. Furthermore, the set of substationary points along with the set of Pareto points in the objective space have been shown in the right-top and left-bottom plots of the figure, respectively.  To achieve an approximation of the Pareto
  front, we applied the \textbf{SUMOPT} method using a uniform grid of $10^{3}$ starting points from the interval $[0, 2\pi]$, and the optimality tolerance $\rho=10^{-3}$. One can see from the bottom-middle plot of Figure~\ref{Fig1'} that \textbf{SUMOPT} has successfully generated the entire Pareto elements.
 
 \begin{figure}[h]
 	\centering
 	\includegraphics[width=\textwidth]{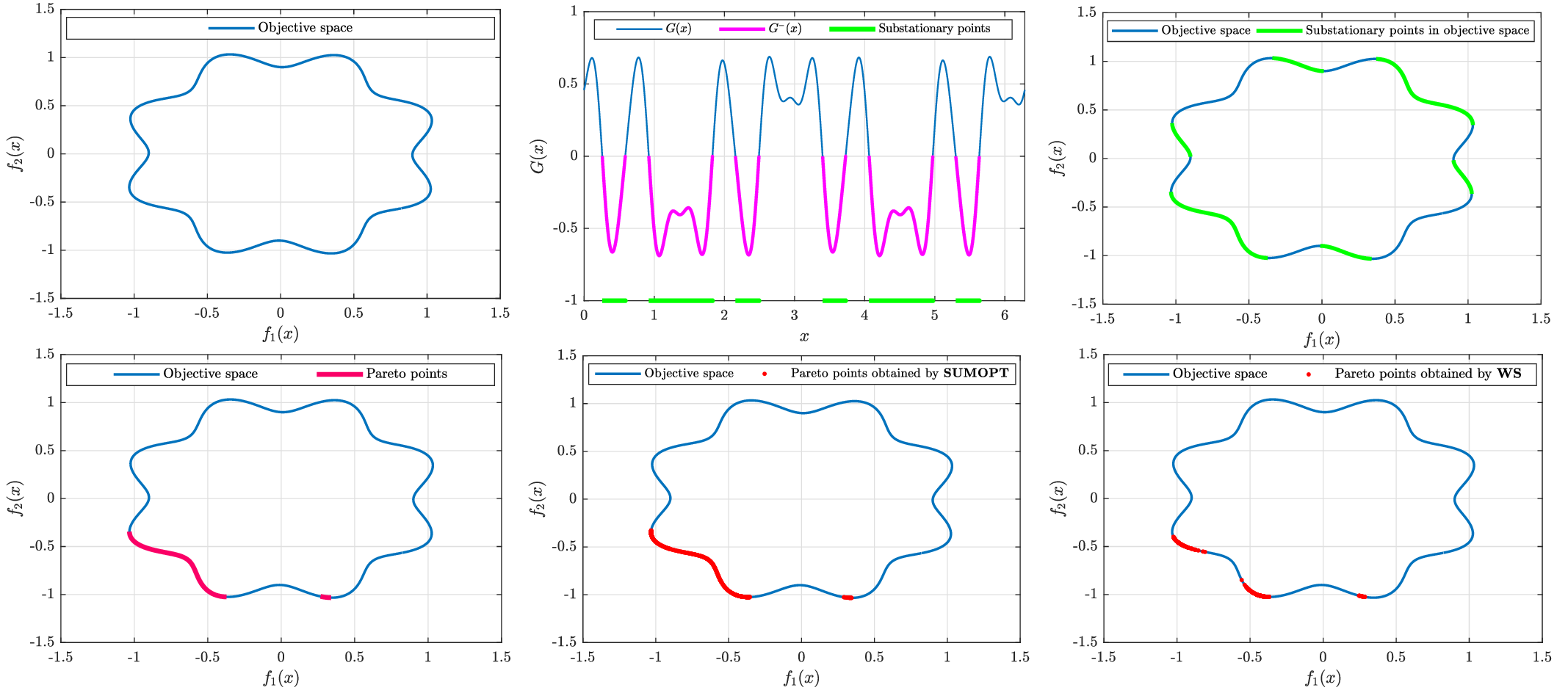}
 	\caption{Top: details of the considered test problem \eqref{FL}. Bottom: Pareto points and obtained approximation of Pareto elements by \textbf{SUMOPT} and \textbf{WS} methods.}\label{Fig1'}
 \end{figure}
 
 For problem \eqref{FL}, we consider the weighted sum problem, which is formulated as follows:
 
 \begin{equation}\label{WS}
 \min_{x \in \mathbb{R}} \lambda f_1(x) + (1-\lambda) f_2(x),
 \end{equation}
  where the weight $\lambda$ varies in the interval $[0,1]$. To solve this problem for various weights, we selected $\lambda$ from a uniform grid of $10^{3}$ points within the interval $[0,1]$. Next, for each $\lambda$,  problem \eqref{FL} was solved using the \emph{patternsearch} solver in \textsc{Matlab}. The obtained results are depicted in the bottom-right plot of Figure~\ref{Fig1'}. It is noteworthy that the \textbf{WS} technique fails to generate a significant portion of the Pareto front.

 \subsection{Sparse solution to an underdetermined linear system}
 In this experiment, we consider the following linear system
 \begin{equation}\label{Lin-Sys}
 \bm A\bm x=\bm b,
 \end{equation} 
 in which, $\bm A\in\mathbb{R}^{m\times n}$  with $m<n$, and $\bm b\in\mathbb{R}^m$. This system is referred to as underdetermined and typically has an infinite number of solutions. Such systems appear in various practical problems, including data fitting, machine learning, denoising, and image processing. The main goal is to find an approximate solution to linear system \eqref{Lin-Sys} that achieves the highest level of sparsity. The classic approach to formulate this problem is as follows:
 \begin{align}\label{F_MOP0}
 \min\,\,  \lVert\bm A\bm x- \bm b\rVert_2^2 +\lambda \lVert \bm x \rVert_1,   \quad \text{s.t.} \quad \bm x\in\mathbb{R}^n,
 \end{align}
  where $\lVert \bm x \rVert_1$ denotes the $l_1$ norm of the vector $\bm x\in\mathbb{R}^n$, and $\lambda>0$ is a positive parameter. However, choosing an appropriate value for the parameter $\lambda$ is a delicate and challenging issue in effectively solving this problem \cite{Beck}. In this regard, one alternative interesting approach is to formulate the problem by the following bi-objective optimization problem: 
 \begin{align}\label{F_MOP}
 	\min\,\, (\lVert \bm x \rVert_1, \lVert\bm A\bm x- \bm b\rVert_2^2 ) \quad \text{s.t.} \quad \bm x\in\mathbb{R}^n.
 \end{align}
  This problem is convex, but nonsmooth due to the use of $l_1$ norm. To create an instance of problem~\eqref{F_MOP}, we randomly generated the matrix $\bm A\in\mathbb{R}^{50\times100}$ and vector $\bm b\in\mathbb{R}^{50}$. The \textbf{SUMOPT} was then applied to this problem using five randomly generated starting points from a unit box, and the optimality tolerance $\rho=10^{-2}$. For each run, the obtained results have been reported in Table \ref{Table3}.
  
  \begin{table}[h]
  	\caption{Results of the \textbf{SUMOPT} method on problem \eqref{F_MOP}, for five randomly generated starting points.}\label{Table3}
  	\resizebox{\textwidth}{!}{%
  		
  		\begin{tabular}{lllllllllllllll}
  			\toprule[1.5pt]
  			& \multicolumn{1}{c}{\rule{0pt}{3ex}Run number} &  & \multicolumn{1}{c}{$\lVert \bm x^*\rVert_1$} &  & \multicolumn{1}{c}{$\lVert \bm A \bm x^*-\bm b\rVert_2^2$} &  & \multicolumn{1}{c}{\#Zero} &  & \multicolumn{1}{c}{Time (s)} &  & \multicolumn{1}{c}{\#Fun} &  & \multicolumn{1}{c}{\#Sub} &  \\ 
  			\cline{2-2}\cline{4-4}\cline{6-6}\cline{8-8}\cline{10-10}\cline{12-12}\cline{14-14}
  			& \multicolumn{1}{c}{\rule{0pt}{3ex}\texttt{1}} &  & \multicolumn{1}{c}{\texttt{1.7632}} &  & \multicolumn{1}{c}{\texttt{1.2932}} &  & \multicolumn{1}{c}{\texttt{76}} &  & \multicolumn{1}{c}{\texttt{4.7154}} &  & \multicolumn{1}{c}{\texttt{14788}} &  & \multicolumn{1}{c}{\texttt{748}} &  \\ 
  			& \multicolumn{1}{c}{\texttt{2}} &  & \multicolumn{1}{c}{\texttt{0.4483}} &  & \multicolumn{1}{c}{\texttt{5.4864}} &  & \multicolumn{1}{c}{\texttt{91}} &  & \multicolumn{1}{c}{\texttt{14.6328}} &  & \multicolumn{1}{c}{\texttt{50002}} &  & \multicolumn{1}{c}{\texttt{2189}} &  \\ 
  			& \multicolumn{1}{c}{\texttt{3}} &  & \multicolumn{1}{c}{\texttt{0.7555}} &  & \multicolumn{1}{c}{\texttt{2.8640}} &  & \multicolumn{1}{c}{\texttt{92}} &  & \multicolumn{1}{c}{\texttt{8.8045}} &  & \multicolumn{1}{c}{\texttt{30008}} &  & \multicolumn{1}{c}{\texttt{1367}} &  \\ 
  			& \multicolumn{1}{c}{\texttt{4}} &  & \multicolumn{1}{c}{\texttt{1.1938}} &  & \multicolumn{1}{c}{\texttt{1.9667}} &  & \multicolumn{1}{c}{\texttt{82}} &  & \multicolumn{1}{c}{\texttt{3.9671}} &  & \multicolumn{1}{c}{\texttt{13542}} &  & \multicolumn{1}{c}{\texttt{680}} &  \\ 
  			& \multicolumn{1}{c}{\texttt{5}} &  & \multicolumn{1}{c}{\texttt{2.1281}} &  & \multicolumn{1}{c}{\texttt{1.0854}} &  & \multicolumn{1}{c}{\texttt{85}} &  & \multicolumn{1}{c}{\texttt{4.7224}} &  & \multicolumn{1}{c}{\texttt{16182}} &  & \multicolumn{1}{c}{\texttt{797}} &  \\ 
  			\toprule[1.5pt]
  		\end{tabular}
  	}
  \end{table}

   In this table, $\bm x^*\in\mathbb{R}^{100}$ is the obtained approximate Pareto point. Additionally, \#Fun and \#Sub denote the total number of function and subgradient evaluations during the minimization process, respectively. Furthermore, \#Zero indicates the number of components in $\bm x^*$ with absolute values less than $10^{-3}$. Indeed, we define a zero component as any having an absolute value less than $10^{-3}$. For all runs, \textbf{SUMOPT} reached the desired accuracy within less than fifteen seconds of CPU time. According to the table, we can see that for the second and third runs, we obtained solutions with fewer non-zero elements, but this comes at a cost of increased residual $\lVert \bm A\bm x^*-\bm b\rVert_2^2$. Table \ref{Table3} also highlights an interesting trade-off: the solution with the smallest $l_1$ norm has the largest residual, and the solution with the smallest residual has the largest $l_1$ norm. Regarding the evaluation of functions and subgradients, we can observe that the number of function evaluations exceeds the number of subgradient evaluations. This is because the \textbf{SUMOPT} algorithm uses a limited backtracking line search to eliminate objectives that do not require computation of a new subgradient after a null step.

 \subsection{A comparison with the MPB method}
 
 This experiment aims to compare the efficiency of the proposed \textbf{SUMOPT} algorithm with the \textbf{MPB} method in terms of the number of function and subgradient evaluations. To this end, we applied these methods to a collection of fifteen nonsmooth convex and nonconvex multiobjective test problems presented in Table \ref{Table0}. Regarding the \textbf{MPB} method, we used the \textsc{Fortran} code from \cite{MPB-Techreport} with the default parameters. For the \textbf{SUMOPT} method,  we worked with the aforementioned parameters with the optimality tolerance $\rho=10^{-3}$. 
 
 Concerning the set of starting points, following \cite{Multi-sub4}, we considered a $13\times13$ uniform grid of the region $[-3, 3]\times[-3, 3]\subset\mathbb{R}^2$ as follows:
 $$\mathbf{S}:=\{ (x_1,x_2)\in\mathbb{R}^2 \,\, : \,\, x_1=-3+0.5i, \,\, x_2=-3+0.5 j, \,\, i,j=0,1,\ldots,12 \}.$$
 Indeed, we ran each problem 169 times using the set of starting points $\mathbf{S}$, and the total number of evaluations is reported in Table \ref{Table5}. It should be noted that \textbf{MPB} and \textbf{SUMOPT} have their own different first-order optimality condition. Yet, by this choice of parameters, both methods successfully reached the desired accuracy in all runs, and the observed results are distinguishable enough so that one can highlight the prominent feature of each. 
 
  In Table \ref{Table5}, for \textbf{SUMOPT}, ``Iter'' represents the total number of iterations of Algorithm \ref{Alg3} during the runs, while for \textbf{MPB}, it denotes the total number of serious steps performed across the runs. Additionally, \#Fun and \#Sub indicate the total number of function and subgradient evaluations throughout the runs. For each objective, \textbf{MPB} evaluates the function value and one arbitrary subgradient at each trial point. Consequently, the number of function and subgradient evaluations is always the same in this method.
 
As observed from the table, for all test problems, \textbf{MPB} requires fewer function evaluations than our proposed method. This can be attributed to the fact that \textbf{SUMOPT} performs a limited backtracking line search to exclude objectives for which new subgradient computations are not required after a null step. However, \textbf{SUMOPT} is more efficient than \textbf{MPB} in terms of subgradient evaluations, as \textbf{MPB} computes a new subgradient for all objectives after each null step, whereas \textbf{SUMOPT} avoids unnecessary subgradient computations. Regarding iterations, the same pattern as for the number of subgradient evaluations occurs. In summary, \textbf{SUMOPT} achieves better subgradient evaluation efficiency compared to \textbf{MPB}, though \textbf{MPB} requires fewer function evaluations through its avoidance of backtracking steps.

 \begin{table}
 	\caption{Results of the \textbf{SUMOPT} and \textbf{MPB} methods on the set of problems considered in Table \ref{Table0}.}\label{Table5}
 	\resizebox{\textwidth}{!}{%
 	\begin{tabular}{lclclclcllclclcl}
 		\toprule[1.5pt]
 		&  &  & \multicolumn{5}{c}{\rule{0pt}{3ex}\textbf{SUMOPT}} &  &  & \multicolumn{5}{c}{\textbf{MPB}} &  \\ 
 		\cline{4-8}\cline{11-16}
 		& \rule{0pt}{3ex}P &  & Iter &  & \#Fun &  & \#Sub &  &  & Iter &  & \#Fun &  & \#Sub &  \\ 
 		\cline{2-2}\cline{4-4}\cline{6-6}\cline{8-8}\cline{11-11}\cline{13-13}\cline{15-15}
 		& \rule{0pt}{3ex}\texttt{1} &  & \texttt{1085} &  & \texttt{20560} &  & \texttt{2704} &  &  & \texttt{1723} &  & \texttt{4426} &  & \texttt{4426} &  \\ 
 		& \texttt{2} &  & \texttt{878} &  & \texttt{13902} &  & \texttt{2367} &  &  & \texttt{1706} &  & \texttt{4528} &  & \texttt{4528} &  \\
 		& \texttt{3} &  & \texttt{711} &  & \texttt{11030} &  & \texttt{1831} &  &  & \texttt{1894} &  & \texttt{4454} &  & \texttt{4454} &  \\ 
 		
 		& \texttt{4} &  & \texttt{660} &  & \texttt{7842} &  & \texttt{1973} &  &  & \texttt{1046} &  & \texttt{2634} &  &\texttt{2634} &  \\ 
 		& \texttt{5} &  & \texttt{1724} &  & \texttt{30006} &  & \texttt{3761} &  &  & \texttt{2950} &  & \texttt{7332}& &\texttt{7332} &  \\ 
 		& \texttt{6} &  & \texttt{1086} &  & \texttt{18854} &  & \texttt{2655} &  &  & \texttt{2609} &  & \texttt{6842} &  & \texttt{6842} &  \\ 
 		& \texttt{7} &  & \texttt{1205} &  & \texttt{20300} &  & \texttt{2777} &  &  & \texttt{1844} &  & \texttt{4068} &  & \texttt{4068} &  \\  
 		& \texttt{8} &  & \texttt{764} &  & \texttt{13146} &  & \texttt{2188} &  &  & \texttt{869} &  & \texttt{2118} &  & \texttt{2118} &  \\ 
 		& \texttt{9} &  & \texttt{1247} &  & \texttt{24144} &  & \texttt{3050} &  &  & \texttt{1795} &  & \texttt{4352} &  & \texttt{4352} &  \\ 
 		& \texttt{10} &  & \texttt{944} &  & \texttt{17314} &  & \texttt{2394} &  &  & \texttt{916} &  & \texttt{2278} &  &\texttt{2278} &  \\
 		& \texttt{11} &  & \texttt{442} &  & \texttt{9050} &  & \texttt{2380} &  &  & \texttt{1078} &  & \texttt{3972} &  & \texttt{3972} &  \\ 
 		& \texttt{12} &  & \texttt{853} &  & \texttt{19299} &  & \texttt{3516} &  &  & \texttt{2241} &  & \texttt{11733} &  & \texttt{11733} &  \\ 
 		& \texttt{13} &  & \texttt{485} &  & \texttt{14004} &  & \texttt{3386} &  &  & \texttt{1006} &  & \texttt{4904} &  & \texttt{4904} &  \\ 
 		& \texttt{14} &  & \texttt{654} &  & \texttt{18234} &  & \texttt{4289} &  &  & \texttt{1253} &  & \texttt{9088} &  & \texttt{9088} &  \\
 		 & \texttt{15} &  & \texttt{575} &  & \texttt{22721} &  & \texttt{4878} &  &  & \texttt{819} &  & \texttt{5070} &  & \texttt{5070} &  \\
 		\toprule[1.5pt]
 	\end{tabular}
 }
 	
 \end{table}

\section{Conclusion} \label{Sec7}

We have developed a descent subgradient method for finding a Clarke substationary point of a general multiobjective optimization problem, whose objectives are locally Lipschitz. As opposed to the multiobjective proximal bundle method proposed by M\"{a}kel\"{a} et al., we do not turn the main problem into a single-objective one by using an improvement function. Instead, in order to find an effective descent direction for the main problem, the behavior of each individual objective is separately studied. This framework allows us to make use of a limited backtracking line search, which helps us determine the objectives for which improving the current approximation of the $\varepsilon$-subdifferential set is unnecessary. This strategy reduced the number of subgradient evaluations during the optimization process.

 Regarding the enrichment process of the current collection of the subgradient information, the proposed method shares a common feature with bundle-type methods, in that the $\varepsilon$-subdifferential sets are sequentially improved. In this regard, we presented  a new variant of Mifflin's line search that provides a non-redundant subgradient, which significantly modifies the current working set associated with each objective in the improvement process. We proved that such a subgradient element can be found through a finite procedure by establishing a relationship between the limited backtracking line search and the proposed Mifflin's line search. We studied the global convergence of the method to a Clarke substationary point, which strengthened the results derived in the method proposed by Gebken and Peitz.
 
 In our numerical tests, we assessed the typical behavior of our algorithm, \textbf{SUMOPT}, and observed the efficiency of the proposed limited backtracking line search in excluding objective functions from the improvement process. Meanwhile, the proposed Mifflin's line search showed its ability in modifying the inefficient search directions, after each null step. Next, we showed the efficiency of our method in generating a uniform approximation of the Pareto fronts. For finding a sparse solution of an underdetermined linear system, our method revealed a trade-off between the sparsity of the solution and the value of the remainder along the Pareto front of the problem. Eventually, our comparative results based on fifteen nonsmooth test problems showed that \textbf{SUMOPT} is a big rival to bundle-like methods, especially in case the computation of subgradients is a time-consuming process.


\begin{thebibliography}{10}
	\providecommand{\url}[1]{#1}
	\csname url@samestyle\endcsname
	\providecommand{\newblock}{\relax}
	\providecommand{\bibinfo}[2]{#2}
	\providecommand{\BIBentrySTDinterwordspacing}{\spaceskip=0pt\relax}
	\providecommand{\BIBentryALTinterwordstretchfactor}{4}
	\providecommand{\BIBentryALTinterwordspacing}{\spaceskip=\fontdimen2\font plus
		\BIBentryALTinterwordstretchfactor\fontdimen3\font minus
		\fontdimen4\font\relax}
	\providecommand{\BIBforeignlanguage}[2]{{%
			\expandafter\ifx\csname l@#1\endcsname\relax
			\typeout{** WARNING: IEEEtran.bst: No hyphenation pattern has been}%
			\typeout{** loaded for the language `#1'. Using the pattern for}%
			\typeout{** the default language instead.}%
			\else
			\language=\csname l@#1\endcsname
			\fi
			#2}}
	\providecommand{\BIBdecl}{\relax}
	\BIBdecl
	
	\bibitem{Meitinent}
	K.~Miettinen, \emph{Nonlinear Multiobjective Optimization}.\hskip 1em plus
	0.5em minus 0.4em\relax New York: Springer, 1999.
	
	\bibitem{Ehrgott2005}
	M.~Ehrgott, \emph{Multicriteria optimization}.\hskip 1em plus 0.5em minus
	0.4em\relax New York: Springer, 2005.
	
	\bibitem{MO_Book}
	Y.~Sawaragi, H.~Nakayama, and T.~Tanino, \emph{Theory of Multiobjective
		Optimization}.\hskip 1em plus 0.5em minus 0.4em\relax Waltham: Academic
	Press, 1985.
	
	\bibitem{bagirov2012}
	A.~M. Bagirov, L.~Jin, N.~Karmitsa, A.~Al~Nuaimat, and N.~Sultanova, ``A
	subgradient method for nonconvex nonsmooth optimization,'' \emph{J. Optim.
		Theory Appl.}, vol. 157, pp. 416--435, 2013.
	
	\bibitem{bagirov2010}
	A.~M. Bagirov and A.~N. Ganjehlou, ``A quasisecant method for minimizing
	nonsmooth functions,'' \emph{Optim. Methods Softw.}, vol.~25, no. (1), pp.
	3--18, 2010.
	
	\bibitem{kiwielbook}
	K.~C. Kiwiel, \emph{Methods of Descent for Nondifferentiable
		Optimization}.\hskip 1em plus 0.5em minus 0.4em\relax Berlin:
	Springer-Verlag, 1985.
	
	\bibitem{inexactBundle1}
	W.~Hare, C.~Sagastiz\'{a}bal, and M.~Solodov, ``A proximal bundle method for
	nonsmooth nonconvex functions with inexact information,'' \emph{Comput.
		Optim. Appl.}, vol.~63, no. (1), pp. 1--28, 2016.
	
	\bibitem{Maleknia-Coap}
	M.~Maleknia and M.~Shamsi, ``A new method based on the proximal bundle idea and
	gradient sampling technique for minimizing nonsmooth convex functions,''
	\emph{Comput. Optim. Appl.}, vol.~77, pp. 379--409, 2020.
	
	\bibitem{Maleknia-Oms}
	------, ``A quasi-{N}ewton proximal bundle method using gradient sampling
	technique for minimizing nonsmooth convex functions,'' \emph{Optim. Methods
		Softw}, vol.~37, no.~4, pp. 1415--1446, 2022.
	
	\bibitem{BT-method}
	H.~Schramm and J.~Zowe, ``A version of the bundle idea for minimizing a
	nonsmooth function: conceptual idea, convergence analysis, numerical
	results,'' \emph{SIAM J. Optim.}, vol.~2, no. (1), pp. 105--122, 1992.
	
	\bibitem{Burke2005}
	J.~V. Burke, A.~S. Lewis, and M.~L. Overton, ``{A robust gradient sampling
		algorithm for nonsmooth, nonconvex optimization},'' \emph{SIAM J. Optim.},
	vol.~15, no. (3), pp. 751--779, 2005.
	
	\bibitem{Kiwiel2007}
	K.~C. Kiwiel, ``Convergence of the gradient sampling algorithm for nonsmooth
	nonconvex optimization,'' \emph{SIAM J. Optim.}, vol.~18, no. (2), pp.
	379--388, 2007.
	
	\bibitem{Curtis2012}
	F.~E. Curtis and M.~L. Overton, ``A sequential quadratic programming algorithm
	for nonconvex, nonsmooth constrained optimization,'' \emph{SIAM J. Optim.},
	vol.~22, no. (2), pp. 474--500, 2012.
	
	\bibitem{Maleknia-jota}
	M.~Maleknia and M.~Shamsi, ``A gradient sampling method based on ideal
	direction for solving nonsmooth optimization problems,'' \emph{J. Optim.
		Theory Appl.}, vol. 187, pp. 181--204, 2020.
	
	\bibitem{Geof}
	A.~M. Geoffrion, ``Proper efficiency and the theory of vector maximization,''
	\emph{J. Optim. Theory Appl.}, vol.~22, pp. 618--630, 1968.
	
	\bibitem{MO-job}
	M.~Gravel, I.~M. Martel, R.~Madeau, W.~Price, and R.~Tremblay, ``A
	multicriterion view of optimal ressource allocation in job-shop production,''
	\emph{European J. Oper. Res.}, vol.~61, pp. 230--244, 1992.
	
	\bibitem{MO2}
	J.~Fliege, L.~M.~G. Drummond, and B.~Svaiter, ``Newton's method for
	multiobjective optimization,'' \emph{SIAM J. Optim}, vol.~20, no.~2, pp.
	602--626, 2009.
	
	\bibitem{Multi-sub1}
	J.~Y.~B. Cruz, ``A subgradient method for vector optimization problems,''
	\emph{SIAM J. Optim.}, vol.~23, no.~4, p. 2169–2182, 2013.
	
	\bibitem{Multi-sub2}
	J.~X. D.~C. Neto, G.~J. P.~D. Silva, O.~P. Ferreira, and J.~O. Lopes, ``A
	subgradient method for multiobjective optimization,'' \emph{Comput Optim
		Appl}, vol.~54, p. 461–472, 2013.
	
	\bibitem{Multi-sub3}
	O.~Montonen, N.~Karmitsa, and M.~M. M\"{a}kel\"{a}, ``Multiple subgradient
	descent bundle method for convex nonsmooth multiobjective optimization,''
	\emph{Comput Optim Appl}, vol.~67, no.~1, pp. 139--158, 2018.
	
	\bibitem{Multi-sub4}
	B.~Gebken and S.~Peitz, ``An efficient descent method for locally lipschitz
	multiobjective optimization problems,'' \emph{J. Optim. Theory Appl.}, vol.
	188, no.~3, p. 696–723, 2021.
	
	\bibitem{techreport-multi}
	M.~M. M\"{a}kel\"{a}, N.~Karmitsa, and O.~Wilppu, ``Multiobjective proximal
	bundle method for nonsmooth optimization. {TUCS} technical report {N}o
	1120,'' Turku Centre for Computer Science, Turku, Tech. Rep., 2014.
	
	\bibitem{QUASI-MO}
	S.~Qu, C.~Liu, M.~Goh, Y.~Li, and Y.~Ji, ``Nonsmooth multiobjective programming
	with quasi-{N}ewton methods,'' \emph{European J. Oper. Res.}, vol. 235,
	no.~3, pp. 503--510, 2014.
	
	\bibitem{Multi-proximal1}
	N.~H. Monjezi and S.~Nobakhtian, ``An inexact multiple proximal bundle
	algorithm for nonsmooth nonconvex multiobjective optimization problems,''
	\emph{Ann. Oper. Res.}, vol. 311, no.~2, pp. 1123--1154, 2022.
	
	\bibitem{Multi-proximal2}
	------, ``A proximal bundle-based algorithm for nonsmooth constrained
	multiobjective optimization problems with inexact data,'' \emph{Numerical
		Algorithms}, vol.~89, no.~2, pp. 637--674, 2022.
	
	\bibitem{Bagirov2014}
	A.~M. Bagirov, N.~Karmitsa, and M.~M. M\"{a}kel\"{a}, \emph{Introduction to
		Nonsmooth Optimization}.\hskip 1em plus 0.5em minus 0.4em\relax Springer
	International Publishing, 2014.
	
	\bibitem{techreport}
	L.~Luk\v{s}an and J.~Vl\v{c}ek, ``Test problems for nonsmooth unconstrained and
	linearly constrained optimization,'' Institute of Computer Science, Academy
	of Sciences of the Czech Republic, Prague, Tech. Rep., 2000.
	
	\bibitem{Evans2015}
	L.~C. Evans and R.~F. Gariepy, \emph{Measure Theory and Fine Properties of
		Functions, Revised Edition}.\hskip 1em plus 0.5em minus 0.4em\relax Boca
	Raton: CRC Press, 1992.
	
	\bibitem{Clarke1990}
	F.~H. Clarke, \emph{{Optimization and Nonsmooth Analysis}}.\hskip 1em plus
	0.5em minus 0.4em\relax Philadelphia: SIAM, 1990.
	
	\bibitem{Mifflin-semismooth}
	R.~Mifflin, ``Semismooth and semiconvex functions in constrained
	optimization,'' \emph{SIAM J. Optim.}, vol.~15, no.~6, pp. 959--972, 1977.
	
	\bibitem{Mifflin}
	------, ``An algorithm for constrained optimization with semismooth
	functions,'' \emph{Mathematics of Operations Research}, vol.~2, no. (2), pp.
	191--207, 1977.
	
	\bibitem{bagirov2020}
	A.~M. Bagirov, N.~Karmitsa, and S.~Taheri, \emph{Partitional Clustering via
		Nonsmooth Optimization}.\hskip 1em plus 0.5em minus 0.4em\relax Switzerland:
	Springer Cham, 2020.
	
	\bibitem{Zowe-book}
	J.~Outrata, M.~Kocvara, and J.~Zowe, \emph{Nonsmooth Approach to Optimization
		Problems with Equilibrium Constraints}.\hskip 1em plus 0.5em minus
	0.4em\relax Dordrecht: Springer, 1998.
	
	\bibitem{murdokhovich-Book1}
	B.~S. Mordukhovich, \emph{Variational Analysis and Generalized Differentiation
		I}.\hskip 1em plus 0.5em minus 0.4em\relax Berlin: Springer, 2006.
	
	\bibitem{Rockafellar2004}
	R.~T. Rockafellar and R.~J.-B. Wets, \emph{Variational Analysis}.\hskip 1em
	plus 0.5em minus 0.4em\relax Berlin: Springer, 2004.
	
	\bibitem{Makela_book}
	M.~M. M\"{a}kel\"{a} and P.~Neittaanm\"{a}ki, \emph{Nonsmooth Optimization:
		Analysis and Algorithms with Applications to Optimal Control}.\hskip 1em plus
	0.5em minus 0.4em\relax Singapore: World Scientific Publishing Co., 1992.
	
	\bibitem{MPB-Techreport}
	M.~M. M\"{a}kel\"{a}, ``Multiobjective proximal bundle method for nonconvex
	nonsmooth optimization: {F}ortran {S}ubroutine {MPBNGC} 2.0,'' Depart. Math.
	Inf. Technol. Ser. {B}. Sci. Comput. {B} 13, 2003, Tech. Rep., 2003.
	
	\bibitem{Multi-start1}
	J.~Fliege, L.~M.~G. Drummond, and B.~F. Svaiter, ``Newton's method for
	multiobjective optimization,'' \emph{SIAM J. Optim.}, vol.~20, no.~2, pp.
	602--626, 2009.
	
	\bibitem{Multi-start2}
	M.~A.~T. Ansary and G.~Panda, ``A globally convergent {SQCQP} method for
	multiobjective optimization problems,'' \emph{SIAM J. Optim.}, vol.~31,
	no.~1, pp. 91--113, 2021.
	
	\bibitem{Multi-Indicators}
	C.~Audet, J.~Bigeon, D.~Cartier, S.~L. Digabel, and L.~Salomon, ``Performance
	indicators in multiobjective optimization,'' \emph{European J. Oper. Res.},
	vol. 292, no.~2, pp. 397--422, 2021.
	
	\bibitem{Beck}
	A.~Beck, \emph{First-Order Methods in Optimization}.\hskip 1em plus 0.5em minus
	0.4em\relax Philadelphia: SIAM, 2017.
	
\end{thebibliography}

\end{document}